\documentclass[a4paper,12pt]{article}

\usepackage[T1]{fontenc} 			
\usepackage{lmodern} 				
\usepackage[english]{babel}			
\usepackage{hyperref} 				
\usepackage{paralist}				
\usepackage{graphicx}				
\usepackage{amsmath} 				
\usepackage{amsfonts} 				
\usepackage{amsopn}				
\usepackage{mathtools}				
\usepackage{amsthm} 				
\usepackage{mathabx}				
\usepackage[arrow, matrix, curve]{xy} 		
\usepackage{extarrows}				
\usepackage{mathtools}				

\newcommand{\sbull}{{\scriptscriptstyle \bullet}}
\DeclareMathOperator{\id}{id}
\DeclareMathOperator{\im}{im}

\DeclareMathOperator{\Image}{im}
\DeclareMathOperator{\GL}{GL}

\DeclareMathOperator{\pr}{pr}

\newcommand{\bl}{{\scriptscriptstyle \bullet}}

\theoremstyle{plain}
\newtheorem{satz}{Satz}[section]		
\newtheorem{theorem}[satz]{Theorem}
\newtheorem{lemma}[satz]{Lemma}
\newtheorem{remark}[satz]{Remark}

\newtheorem{satz/definition}[satz]{Satz/Definition}
\newtheorem{lemma/definition}[satz]{Lemma/Definition}

\newtheorem{theoo}{Theorem}

\theoremstyle{definition}
\newtheorem{definition}[satz]{Definition}

\newtheorem{definiti}[theoo]{Definition}
\newtheorem{probleme}[theoo]{Problem}
\newtheorem*{convention}{Convention}
\newtheorem*{problem2}{Problem}

\allowdisplaybreaks				
\setdefaultitem{\textbullet}{-}{}{}		

\title{The Frobenius theorem for Banach distributions on infinite-dimensional manifolds and applications in infinite-dimensional Lie theory}
\author{Jan Milan Eyni}
\date{}
\allowdisplaybreaks				
\setdefaultitem{\textbullet}{-}{}{}		

\begin{document}
\maketitle

\thispagestyle{empty}

\begin{abstract}
We prove a Frobenius theorem for Banach distributions on manifolds that are modelled over locally convex spaces. Moreover, we recall how Frobenius theorems can be applied to infinite-dimensional Lie groups and obtain, that given a Lie subalgebra of the Lie algebra of a Lie group that is modelled over a locally convex space that admits an exponential map, the Lie subalgebra is integrable if it is complemented as a topological vector space and a Banach space with the induced topology.
\end{abstract}

{\footnotesize
{\bf Jan Milan Eyni},
Universit\"{a}t Paderborn,
Institut f\"{u}r Mathematik,
Warburger Str.\ 100,
33098 Paderborn, Germany;
\,{\tt janme@math.upb.de}\\[2mm]

\section*{Introduction}
The Frobenius theorem for finite-dimensional manifolds is a classical result of differential geometry. Furthermore the case of Banach manifolds and its applications to Banach Lie groups are well known, and presented for instance in \cite[Chapter 1]{Lang}. However, there are interesting groups that are not Banach Lie groups. For example, it is known that the diffeomorphism group of a compact finite-dimensional manifold is a Lie group modelled over a locally convex space but not a Banach Lie group. To explain what a manifold modelled over a locally convex space is one has to define a differential calculus on locally convex spaces. There are different approaches for such a differential calculus. One popular approach is the convenient calculus of Fr{\"o}licher, Kriegl and Michor (\cite{Kriegel u. Michor}). In this context a global function $f\colon E\rightarrow F$ is called convenient smooth if for all smooth curves $\gamma \colon \mathbb{R} \rightarrow E$ the curve $f\circ \gamma \colon \mathbb{R}\rightarrow F$ is smooth. Another popular approach to $C^k$-maps used in this article, is Keller's $C^k_c$-theory, as popularized in \cite{Milnor}.
\begin{definiti}
If $E$ and $F$ are locally convex spaces and $U\subset E$ is open, then a continuous map $f\colon U \rightarrow F$ is called a $C^1$-map (respectively of class $C^1$) if the directional derivative
\begin{align*}
\lim\limits_{h\rightarrow 0} \frac{f(p+hv)-f(p)}{h} =: df(p,v)
\end{align*}
exists for all $p\in U$ and $v\in E$ and the function $df\colon U\times E \rightarrow F,~(p,v)\mapsto df(p,v)$ is continuous. Moreover, given $r\in \mathbb{N}$ the function $f$ is called a $C^r$-map (respectively of class $C^r$) if $f$ is a $C^1$-map and $df$ is a $C^{r-1}$-map.
\end{definiti}
Now it is clear what a manifold modelled over a locally convex space should be. We simply use the finite-dimensional definition with locally convex spaces instead of finite-dimensional spaces and the $C^k$-differential calculus just described instead of the finite-dimensional one. A Lie group modelled over a locally convex space is a group that is a manifold modelled over a locally convex space such that the group operations are smooth. For more details the reader may consider for example \cite{Milnor} or \cite{Neeb}. There is also a comprehensive book about infinite-dimensional Lie groups in preparation with \cite{Gloeckner u. Neeb}.

In 2006 Neeb stated the following problem in \cite{Neeb}:
\begin{probleme}\label{EinlProb}
Let $G$ be a regular Lie group. Is every closed Lie subalgebra $\mathfrak{h}\subseteq L(G)$  with finite codimension integrable?  
\end{probleme}
Gl{\"o}ckner proposed in \cite{Gloeckner} that generalisations of the Frobenius theorems to manifolds that are modelled over locally convex spaces can be used to solve this problem. 

Hiltunen showed in \cite{Hiltunen} that certain co-Banach distributions for manifolds that are modelled over locally convex spaces in the sense of Fr{\"o}hlicher-Bucher differential calculus, are locally integrable. This means that given a point in the manifold we find an immersed submanifold containing the point whose tangent spaces are exactly the fibres of the given distribution. This immersed submanifolds are called integral manifolds of the distribution. In \cite{Eyni}, we not only used the methods from \cite{Hiltunen} to obtain a similar result in the setting of Keller's $C^k_c$-theory, but it we also constructed a foliation of the manifold consisting of maximal integral manifolds. These maximal integral manifolds and their local parametrisations by certain charts are actually necessary to solve Problem \ref{EinlProb}. This was done in \cite[Chapter 4]{Eyni}. There, methods from \cite[Chapter 1]{Lang} for the case of Banach Lie groups were generalised to the case of Lie groups modelled over locally convex spaces. Also it was possible to correct a certain inaccuracy of Lang in \cite[7, Chapter VI. Lemma 5.3]{Lang} with the help of Gl{\"o}ckner.

In 2001, Teichmann showed a Frobenius theorem for finite-dimensional distributions on manifolds that are modelled on locally convex spaces in the convenient sense. It was possible to obtain the analogous result in the author's master's thesis \cite{Eyni} in the context of manifolds that are modelled over locally convex spaces in the sense of Keller's $C_c^k$-Theory. Moreover, by using the methods from \cite[Chapter 4]{Eyni} we showed that if the Lie group $G$ in question has an exponential map, then  every finite-dimensional Lie subalgebra $\mathfrak{h}\subseteq L(G)$ is integrable.

Now, it is a natural question to ask if every Lie subalgebra $\mathfrak{h}\subseteq L(G)$ that is complemented as a topological vector subspace and is a Banach space with the induced topology is integrable as well. The answer is yes. In this article we will prove the following theorem.
\begin{theoo}
Let $G$ be a Lie group modelled over a locally convex space and $\mathfrak{h}\subseteq L(G)$ be a Lie subalgebra that is complemented as a topological vector subspace and is a Banach space. If $G$ admits an exponential map, then we can find a Lie group $H$ that is a subgroup of $G$ and an immersed submanifold of $G$ such that $L(H)=\mathfrak{h}$.
\end{theoo}
The main work will be to show a Frobenius theorem for Banach distributions for manifolds that are modelled over locally convex spaces.

The structure of the article is as follow. In Section \ref{Vector distributions of infinite-dimensional manifolds} we recall constructions for distributions on infinite-dimensional manifolds. These results were already stated in \cite{Eyni}. They are mainly easy generalisations of the well known finite-dimensional case from \cite{Warner}. Nevertheless, these constructions are recalled for the convenience of the reader.

The heart of the article, which contains the new results, is Section \ref{Frobenius theorem}. Here, we show a Frobenius theorem for certain Banach distributions on manifolds that are modelled over locally convex spaces. Besides new arguments we use methods from the case where the distribution in question is finite-dimensional (\cite{Teichmann} respectively \cite{Eyni}). Also, we use methods developed in \cite{Stefan}, where Chillingwoth and Stefan work with singular distributions on Banach manifolds. The following theorem (details of which will be explained later) will be obtained there.
\begin{theoo}\label{BanachBanachEinl}
Let $M$ be a $C^r$-manifold modelled over a locally convex space $E$ with $r\geq 2$ and $F$ be a complemented subspace of $E$ such that $F$ is a Banach space and $D$ be an involutiv subbundle of $TM$ with typical fibre $F$. If for $p_0\in M$ there exists an open $p_0$-neighbourhood $U\subseteq M$ and a $C^{r-1}$-vector field $X \colon U\times F \rightarrow TU$  with parameters in $F$ such that
\begin{compactenum}
\item the map $\check{X}\colon F \rightarrow \mathcal{V}(U),~ v\mapsto X(\bl,v)$ is linear,\label{BanachBanach1}
\item we have $\im(X) \subseteq D$,\label{BanachBanach2}
\item the map $X(p_0,\sbull)|^{D_{p_0}} \colon F \rightarrow D_{p_0}$ is an isomorphism of topological vector spaces and\label{BanachBanach3}
\item the vector field $X$ provides a local flow with parameters,\label{BanachBanach4}
\end{compactenum}
then $D$ is a Frobenius distribution. 
\end{theoo}
In Theorem \ref{BanachBanachEinl}, we have to assume that the vector field admits a local flow, because this is not automatic for initial value problems in locally convex spaces. Indeed, it is possible to find linear initial value problems in locally convex spaces that have several solutions, or no solution at all.

In Section \ref{LieTheory}, we recall the constructions from \cite{Eyni} that we used to solve the Problem~\ref{EinlProb}. The only thing that is new in this section is Theorem \ref{satzBanach} and its proof. In \cite{Eyni}, we generalised the results of \cite[Chapter 1]{Lang} from the case of Banach Lie groups to Lie groups that are modelled over locally convex spaces. We recall these results and constructions for the convenience of  the reader.

In the appendix we describe the theory of flows of vector fields on infinite-dimensional manifolds. As mentioned above, we do not have a solution theory for initial value problems, so that it is necessary  to assume the existence of local flows, to construct the global one. All results except of the Lemmas \ref{W} and \ref{EinD}, that come from \cite{Gloeckner u. Neeb}, are straightforward generalisations of the case of Banach manifolds, as presented in \cite[Chapter 1]{Lang}.

\section{Vector distributions of infinite-dimensional manifolds}\label{Vector distributions of infinite-dimensional manifolds}

The definitions, theorems and proofs of this section were already stated in \cite[Chapter 1]{Eyni}. So there are no new results in this section. We only present them for the convenience of the reader. Also, we mention that the results from \cite[Chapter 1]{Eyni} are mainly based on \cite[Chapter 1]{Warner}. But where Warner only discussed finite-dimensional manifolds, we generalised the arguments to infinite-dimensional manifolds modelled over locally convex spaces in the straightforward way.

\begin{convention}
Throughout this section $E$ will be a locally convex space, $r \in \mathbb{N}$ and $M$ a $C^r$-manifold modelled over $E$.
\end{convention}

\begin{definition}
A subset $D\subseteq TM$ is called \emph{vector distribution} or just {\it distribution} of $M$, if for every point $p\in M$ the set $D_p:=D\cap T_pM$ is a subspace of $T_pM$. Important examples for vector distributions are subbundles of $TM$. 
\end{definition}

\begin{definition}
A subset $N\subseteq M$ is called \emph{immersed submanifold} of $M$, if it is a $C^r$-manifold modelled over a closed subspace $F$ of $E$ such that the inclusion $\iota_N^M\colon N\rightarrow M$, $p\mapsto p$ is continuous and given $p\in N$ we find a chart $\varphi\colon U_\varphi\rightarrow V_\varphi$ of $N$ around $p$ and a chart $\psi\colon U_\psi\rightarrow V_\psi$ of $M$ around $p$ such that $U_\varphi\subseteq U_\psi$ and $\psi\circ\iota_N^M\circ \varphi^{-1}=\iota_F^{E}|_{V_\varphi}$.
\end{definition}

\begin{definition}\label{IM-Def}
Let $F\subseteq E$ be a closed subspace of $E$ and $D\subseteq TM$ be a subbundle of $TM$ with typical fibre $F$. A connected immersed submanifold $N\subseteq M$ is called \emph{integral manifold} for $D$, if $T_pM=D_p$ for every $p\in N$. Given $p_0\in M$, we call an integral manifold $N$ containing $p_0$ \emph{maximal} if every other integral manifold $L$ of $D$ that contains $p_0$ is a subset of $N$ and the inclusion map $\iota\colon L\hookrightarrow N$, $p\mapsto p$ is of class $C^r$.
\end{definition}

\begin{remark}\label{Eindeutigkeit}
In the situation of Definition \ref{IM-Def} the maximal integral manifold that contains $p_0$ is obviously unique, whence the ``maximal integral manifold'' is in fact a largest one.
\end{remark}

\begin{definition}\label{Vektordistributionen1}
Let $F\subseteq E$ be a closed subspace and $D\subseteq TM$ be a subbundle of $TM$ with typical fibre $F$. Assume that $F$ is complemented in $E$ say $E=F\oplus H$ topologically with vector subspace $H$ of $E$. A chart $\varphi\colon U_\varphi\rightarrow V_\varphi$ of $M$ is called a \emph{Frobenius chart} for $D$, if there are open sets $V_1\subseteq F$ and $V_2\subseteq H$ such that $V_\varphi=V_1\times V_2$ and for $\overline{y}\in V_2$ the submanifold 
\begin{align}\label{S_y}
S_{\overline{y}}:=\{\varphi^{-1}(x,\overline{y}):x\in V_{1}\}
\end{align}
is an integral manifold of $D$. If $M$ admits an atlas of Frobenius charts for $D$, we call $D$ a \emph{Frobenius distribution}.
\end{definition}

\begin{lemma}\label{Vektordistributionen2}
Let $F\subseteq E$ be a closed subspace of $E$ and $D\subseteq TM$ be a subbundle of $TM$ with typical fibre $F$. Assume that $F$ is complemented in $E$ with topological vector complement $H\subseteq E$. If $\varphi\colon U_\varphi\rightarrow V_\varphi$ is a Frobenius chart where $V_\varphi=V_1 \times V_2$ with open sets $V_1 \subseteq F$ and $V_2 \subseteq H$ and $L\subseteq M$ is an integral manifold for $D$ with $L\subseteq U_\varphi$, then there exists a $\overline{y}\in V_2$ such that $L\subseteq S_{\overline{y}}$ for $S_{\overline{y}}$ like in (\ref{S_y}). 
\end{lemma}
\begin{proof}
Let $p\in L$ and $\overline{y}:=\pr_{2}\circ\varphi(p)$. To show $L\subseteq S_{\overline{y}}$ we proof that $\pr_{2}\circ\varphi|_{L}$ is constant. The integral manifold $L$ is connected, hence it is enough to show $d(\pr_{2}\circ \varphi)|_{T_{q}L})=0$ for all $q\in L$. From $T_{q}L=D_{q}=T_{q}S_{\overline{y}}$ we conclude $d(\pr_{2}\circ \varphi)|_{T_{q}L}=d(\pr_{2}\circ \varphi)|_{T_{q}S_{\overline{y}}}=0$.
\end{proof}

\begin{lemma}\label{IUMF}
If $F$ is a complemented subspace of $E$ and $N$ is an immersed submanifold of $M$ modelled over $F$, then for each point $p_0\in N$ we find a submanifold $S$ of $M$ such that $S$ contains $p_0$ and is open in $N$.
\end{lemma}
\begin{proof}
Let $\varphi\colon U_\varphi\rightarrow V_\varphi$ be a chart of $M$ around $p_0$, with $\varphi(p_0)=0_E$ and $\psi\colon U_\psi\rightarrow V_\psi$ a chart of $N$ around $p_0$ such that $U_{\psi}\subseteq U_\varphi$ and $\varphi\circ \iota_N^M\circ\psi^{-1}=\iota_F^E|_{V_\psi}$. Hence $\varphi\circ\psi^{-1}=\iota_F^E|_{V_\psi}$ and so
\begin{align}
\varphi|_{U_\psi}=\iota_F^E|_{V_\psi}\circ \psi.
\label{GL_iota}
\end{align}
Let $H$ be the topological vector complement of $F$ in $E$. We assume that there exists an open subset $V_\varphi^{1}$ of $F$ and an open subset $V_\varphi^{2}$ of $H$ such that $V_\varphi=V_\varphi^{1}\times V_\varphi^{2}$. With (\ref{GL_iota}) we conclude $\varphi(U_\psi)=V_\psi\times \{0\}$ and get
\begin{align}
U_\psi=\varphi^{-1}(V_\psi\times \{0\}).
\label{GLiota2}
\end{align}
Moreover (\ref{GL_iota}) leads to $V_\psi\subseteq V_\varphi^{1}$ and we assume $V_\psi=V_\varphi^{1}$. With (\ref{GLiota2}) we conclude $S:=\{\varphi^{-1}(x,0):x\in V_\varphi^{1}\}=U_\psi$ whence $S$ has the required properties.
\end{proof}

\begin{theorem}
Let $F$ be a closed subspace of $E$ and $D\subseteq TM$ be a subbundle of $TM$ with typical fibre $F$. If $D$ is a Frobenius distribution, then given $p_0\in M$ there exists a maximal integral manifold that contains $p_0$. According to Remark \ref{Eindeutigkeit} this maximal integral manifold is unique.
\label{Existenz-MIM}
\end{theorem}
\begin{proof}
Because $D$ is a Frobenius distribution, $F$ is complemented in $E$ with a topological vector complement $H$. For each $p\in M$ we can find a Frobenius chart $\phi_p\colon U_{\phi_p}\rightarrow V_{\phi_p}=V_{\phi_p}^{(1)}\times V_{\phi_p}^{(2)}$ where $V_{\phi_p}^{(1)}$ is open in $F$ and $V_{\phi_p}^{(2)}$ is open in $H$ and $\phi_p(p)=0$. The submanifolds $S_p:=\{\phi_{p}^{-1}(x,0):x\in V^{(1)}_{\phi_{p}}\}$ are integral manifolds of $D$. We now write $K$ for the set of all points for which we can find a piece wise $C^r$-curve from $p_0$ to $p$ such that its derivative at each time lies in $D$. Given $p\in K$ also $S_p\subseteq K$. We give $K$ the final topology with respect to the inclusions $\iota_{S_{p}}^{K}\colon S_{p}\hookrightarrow K$ for $p\in K$. 

We show that every open set $U\subseteq S_p$ with respect to the induced topology from $M$ is also open in $K$. To this end let $p\in K$ and $U\subseteq S_p$ be open. Let $q\in K$ be another point in $K$. We show that $U\cap S_q$ is open in $S_q$. 

We assume $U\cap S_q\neq \emptyset$. Let $\overline{p}\in U\cap S_q$ (especially $\overline{p} \in S_p$). There is an open subset $U^M$ of $M$ with $U=S_{p}\cap U^{M}$. The set $U_2:=S_q\cap U^M$ is open in $S_q$ and contains $\overline{p}$. We write $(U_2\cap U_{\phi_p})_{\overline{p}}$ for the connected component of $\overline{p}$ in $U_2\cap U_{\phi_p}$. Because $S_q$ is an integral manifold of $D$, also $(U_2\cap U_{\phi_p})_{\overline{p}}$ is an integral manifold of $D$. Using Lemma \ref{Vektordistributionen2}, $(U_2\cap U_{\phi_p})_{\overline{p}} \subseteq U_{\phi_p}$ and $\overline{p}\in U_2$ we conclude $(U_2\cap U_{\phi_p})_{\overline{p}} \subseteq S_p$. Hence $U_{2}\cap S_{p} \supseteq (U_{2}\cap U_{\phi_{p}})_{\overline{p}}$. And thus $U\cap S_{q}=U^{M}\cap S_{p}\cap S_{q}=S_{p}\cap U_{2} \supseteq (U_{2}\cap U_{\phi_{p}})_{\overline{p}}$. 

The set $U\cap S_{q}$ is open in $S_q$, because $U_{\phi_p}$ is open in $M$, $U_2$ is open in $S_q$ and $\overline{p} \in S_q\cap U$ was arbitrary. We have shown that for each $p\in K$ every open set $U\subseteq S_p$ is also open in $K$. 

The set $\left\{ {\pr}_{1}\circ\phi_{p}|_{S_{p}}:p\in K\right\}$ is a $C^r$-atlas for $K$, because $\iota_{S_p}^K$ is a topological embedding with open image. We make $K$ a $C^r$-manifold using the corresponding maximal $C^r$-atlas. 

Because $K$ is path connected, $K$ is also connected. For every $p\in K$ we get $\phi_{p}\circ\iota_{K}^{M}\circ(\phi_{p}|_{S_{p}})^{-1}=\iota_{F}^{E}$. This implies that $K$ is an immersed submanifold of $M$. Since $S_p$ is open in $K$ for $p\in K$, we get $T\iota_{K}^{M}(T_{p}K)=T\iota_{K}^{M}(T_{p}S_{p})=T_{p}S_{p}=D_{p}$. We conclude that $K$ is an integral manifold for $D$. 

It remains to show that $K$ is also a maximal integral manifold for $D$ that contains $p_0$. To this end let $L$ be an integral manifold for $D$ that contains $p_0$ and $p\in L$. Because $L$ is a connected manifold, we can find a piece wise $C^r$-curve $c\colon[0,1]\rightarrow L$ with $c(0)=p_0$ and $c(1)=p$. We get $c'(t)\in T_{c(t)}L=D_{c(t)}$ for $t$ in an interval on which $c$ is of class $C^r$. Hence $L\subseteq K$. 

It remains to show that $\iota_L^K$ is of class $C^r$. Let $p\in L$. Since $L$ is an immersed submanifold of $M$, we find an open subset $S$ of $L$ that is a submanifold of $M$ and contains $p$. Thus $(S\cap U_{\phi_p})_{p}$ contains $p$, is open in $L$, and is a submanifold of $M$. Therefore  $(S\cap U_{\phi_p})_p$ is an integral manifold of $D$ and is contained in $U_{\phi_p}$. This implies $(S\cap U_{\phi_p})_p \subseteq S_p$. Because $(S\cap U_\phi)_p$ is an open $p$-neighbourhood in $L$ and contained in $S_p$, we only need to show that $\iota_L^K|_{(S\cap U_\phi)_p}^{S_p}$ is a $C^r$-map. This follows directly from $\iota_L^K|_{(S\cap U_\phi)_p}^{S_p}={\id}_M|_{(S\cap U_\phi)_p}^{S_p}$.
\end{proof}
The idea to take $K$ as a candidate for the searched maximal integral manifold comes from \cite[Theorem 1.64]{Warner}. But we modified the proof that $K$ is actually an integral manifold. As mentioned above this was already done in \cite[Chapter 1]{Eyni}.

\begin{definition}
If $F$ is a closed vector subspace of $E$, we call a subbundle $D\subseteq TM$ of $TM$ with typical fibre $F$ \emph{involutiv}, if for every vector fields $X,Y\colon U\rightarrow TM$ on an open set $U\subseteq M$ with $\Image(X)\subseteq D$ and $\Image(Y)\subseteq D$, also $\im([X,Y]) \subseteq D$.
\end{definition}

Using the idea of the proof from \cite[Proposition 1.59]{Warner} we obtain a proof of the following Theorem \ref{notwendig}. As mentioned above this was already done in \cite[Chapter 1]{Eyni}.
\begin{theorem}\label{notwendig}
If $F$ is a closed vector subspace of $E$, $D\subseteq TM$ is a subbundle of $TM$ with typical fibre $F$ and given $p\in M$ we find an integral manifold $N$ for $D$ with $p\in N$, then $D$ is involutiv.
\end{theorem}
\begin{proof}
Let $X_1$ and $X_2$ be vector fields on $M$ with $\Image(X_i)\subseteq D$ for $i\in\{1,2\}$. Let $p\in M$. To see $[X_1,X_2](p)\in D_p$, let $N$ be an integral manifold for $D$ with $p\in N$. There exists an open $p$-neighbourhood $U\subseteq N$ such that $U$ is a submanifold of $M$. Since $X_i(U)\subseteq TU$ and $TU\subseteq TM$ is a submanifold, $X_i|_U$ is of class $C^{r-1}$ also as a map to $TU$. Hence $X_1|_U$ and $X_2|_U$ are $C^{r-1}$-vector fields on $U$ and so is $[X_1|_U,X_2|_U]$. Let $\iota\colon U\rightarrow M$ be the inclusion map. Since $X_1|_U\colon U\rightarrow TU$ is $\iota$-related to $X_1$ and $X_2|_U\colon U\rightarrow TM$ is $\iota$-related to $X_2$ so $[X_1|_U,X_2|_U]$ is $\iota$-related to $[X_1,X_2]$. Hence $[X_1,X_2](p)=[X_1,X_2](\iota(p))=T\iota([X_1|_U,X_2|_U](p))\in T\iota (T_p U)=D_p$.
\end{proof}

\begin{remark}\label{FB-charakter}
Let $F$ be a complemented subspace of $E$ with vector complement $H$ and $D\subseteq TM$ be a subbundle of $TM$ with typical Fibre $F$. For a chart $\phi\colon U\rightarrow V_1\times V_2$ of $M$ and the inclusion $\iota_{\overline{y}}\colon V_1\rightarrow V_1\times V_2$, $x\mapsto (x,\overline{y})$ we get the following equivalences:
\begin{align*}
&\phi \text{ is a Frobenius chart}\\
\Leftrightarrow & (\forall \overline{y}\in V_2) ~ S_{\overline{y}}^\phi=\phi^{-1}(\sbull,\overline{y})(V_1)\text{ is an integral manifold for }D\\
\Leftrightarrow & (\forall \overline{y}\in V_2)(\forall x\in V_1) ~ T_{\phi^{-1}(x,\overline{y})} S_{\overline{y}}^\phi=T_x(\phi^{-1}\circ\iota_{\overline{y}})(\{x\}\times F)= D_{\phi^{-1}(x,\overline{y})}\\
\Leftrightarrow & (\forall p\in U_\phi) ~ d\phi(D_p)=F.
\end{align*}
\end{remark}

\begin{definition}
We call a Distribution $D$ with typical fibre $F$ of $M$ \emph{local Frobenius distribution}, if given $p\in M$ we find a chart $\phi\colon U\rightarrow V$ of $M$ around $p$ with $d\phi(D_p)=F$ such that the subbundle $D_\phi:=T\phi (D\cap TU)$ of $V\times E$ with typical fibre $F$ is a Frobenius distribution. For every subbundle chart $\psi$ of $D$, the map $\psi_\phi:=\psi\circ T\phi^{-1}$ is a subbundle chart for $D_\phi$.
\end{definition}

\begin{lemma}\label{localfb}
A Distribution $D$ for $M$ is a Frobenius distribution if and only if it is a local Frobenius distribution.
\end{lemma}
\begin{proof}
It is clear that a Frobenius distribution is a local Frobenius distribution. 

Let $D$ be a local Frobenius distribution, $p\in M$ and $\phi\colon U\rightarrow V$ a chart of $M$ around $p$ with $\phi(p)=0$ and $d\phi(D_p)=F$. We find a diffeomorphism $\chi\colon V'\rightarrow W$ with $V'$ open in $V$ that is a Frobenius chart for $D_\phi:=T\phi (D\cap TU)$. We calculate 
\begin{align*}
d(\chi\circ\phi)(D_q)=d\chi(T\phi(D_q))=d\chi((D_\phi)_{\phi(q)})=F
\end{align*}
and by Remark \ref{FB-charakter} we see that $\chi\circ\phi$ is a Frobenius chart of $D$.
\end{proof}

\section{The Frobenius theorem for Banach distributions}\label{Frobenius theorem}
As mentioned in the introduction, Teichmann showed a Frobenius theorem for finite-dimensional vector distributions on convenient manifolds that are modelled over locally convex spaces. A similar result for manifolds that are modelled over locally convex spaces in the sense described in the introduction was obtained in \cite[Chapter 2; Theorem 2.6]{Eyni}. The aim of this section and of the whole article is to obtain a Frobenius theorem for Banach distributions on manifolds that are modelled over locally convex spaces. 

In \cite{Stefan} Stefan considers distributions of Banach manifolds that are not necessarily subbundles of the tangent bundle but each fibre of the distributions in question is a Banach space. Our proof is inspired by the proofs of \cite[Section 4]{Stefan} and \cite[Theorem 2]{Teichmann} respectively \cite[Theorem 2.6]{Eyni}. But as Stefan considers Banach manifolds, we are interested in manifolds that are modelled over locally convex spaces. So one of the main problems will be that we have no solution theory for initial value problems in locally convex spaces.

The idea to generalise the methods used in \cite{Stefan}  received the author from Gl{\"o}ckner.

\begin{convention}
Throughout this section $E$ will be a locally convex space, $r \in \mathbb{N}$ and $M$ a $C^r$-manifold modelled over $E$.
\end{convention}

\begin{definition}
Let $N$ be a $C^r$-manifold, $X\colon N\rightarrow TN$ be a $C^{r-1}$-vector field of $N$ and $f\colon M\rightarrow N$ be a diffeomorphism. In this situation we define the $C^{r-1}$-vector field $f^{\ast}X:=Tf^{-1}\circ X\circ f$ of $M$.
\end{definition}

The following easy Lemmas \ref{ENDL3} and \ref{ENDL4} have also been stated in \cite{Eyni}. They are also more or less clear and straightforward generalisation of the finite-dimensional case.

\begin{lemma}\label{ENDL3}
If $X\colon M\rightarrow TM$ is a $C^{r-1}$-vector field that provides a local flow and  $t\in \mathbb{R}$ we get $(\Phi_{t}^X)^{\ast}X = X$ on $\Omega^X_t$, where we write $\Phi^X \colon \Omega^X \rightarrow M$ for the global flow of $X$ like in Definition \ref{GlobFluss}.
\end{lemma}
\begin{proof}
We assume $\Omega_t \neq \emptyset$. Given $p\in \Omega^X_t$ we have
\begin{align*}
&X(\Phi^X_{t}(p))=\frac{d}{ds}\Big|_{s=0}\Phi^X_{t+s}(p) =\frac{d}{ds}\Big|_{s=0}\Phi^X_{t}(\Phi^X_{s}(p)) =T\Phi^X_{t}\left(\frac{d}{ds}\Big|_{s=0}\Phi^X_{s}(p)\right)\\
=&T\Phi^X_{t}(X(p)).
\end{align*}
The assertion follows directly.
\end{proof}

\begin{lemma}\label{ENDL4}
Given a  $C^r$-manifold $N$, a diffeomorphism $\phi\colon M\rightarrow N$ and two vector fields $X,Y\colon N\rightarrow TN$ we get $\phi^{\ast}[X,Y]=[\phi^{\ast}X,\phi^{\ast}Y]$.
\end{lemma}
\begin{proof}
From $T\phi^{-1}\circ X\circ \phi=\phi^{\ast}X$ we conclude that $X$ and $\phi^\ast X$ are $\phi^{-1}$-connected. An analogous statement holds for $Y$. Hence also $[X,Y]$ is $\phi^{-1}$-connected to $[\phi^{\ast}X,\phi^{\ast}Y]$ and so $T\phi^{-1}\circ[X,Y]=[\phi^{\ast}X,\phi^{\ast}Y]\circ\phi^{-1}$.
\end{proof}

The following Lemma \ref{AblLieKlam} is a straightforward generalisation of the finite-dimensional case.
\begin{lemma}\label{AblLieKlam}
If $X,Y \colon M \rightarrow TM$ are $C^{r-1}$-vector fields and $X$ provides a local flow, then we have 
$\frac{d}{ds}\big|_{s=0} \left( (\Phi_s^X)^\ast Y(p) \right) = [X,Y](p)$
for all $p \in \Omega^X$.
\end{lemma}
\begin{proof}
It is enough to show the assertion in the local case. Let $U \subseteq E$ be an open subset and $f,g\colon U \rightarrow E$ be $C^{r-1}$-maps such that $f$ provides a local flow. We write $\Phi\colon \Omega \rightarrow U$ for the global flow of $f$. For $p \in U$ we calculate
\begin{align*}
&\frac{d}{ds}\Big|_{s=0} (\Phi_s^\ast g (p))= \frac{d}{ds}\Big|_{s=0} d\Phi_{-s}\left(\Phi_s(p) , g(\Phi_s(p))\right)\\
=& \frac{d}{ds}\Big|_{s=0} d\Phi (-s , \Phi(s,p); 0, g(\Phi(s,p))\\
=& d_1(d\Phi)\left(0 , \Phi(0,p) , 0, g (\Phi(0,p)); -1,\frac{d}{ds}\Big|_{s=0} \Phi (s,p)\right) \\
&+ d\Phi \left(0,\Phi(0, p); 0 , \frac{d}{ds} \Big|_{s=0} g(\Phi(s,p)) \right)\\
=& d_1(d\Phi)(0 , p , 0, g (p); -1, 0) + d_1(d\Phi)(0 , p , 0, g (p); 0, f(p)) \\
&+ d\Phi(0,p; 0,dg(p,f(p))) \\
=& -\frac{d}{dt}\Big|_{t=0} d\Phi(t , p ; 0, g (p)) +  \frac{d}{dt}\Big|_{t=0} d\Phi(0,p+tf(p);0,g(p))\\
&+d_2\Phi(0,p;dg(p,f(p)))\\
=& -\frac{d}{dt}\Big|_{t=0} \frac{d}{ds}\Big|_{s=0} \Phi(t , p + sg (p)) + \frac{d}{dt}\Big|_{t=0} d\Phi_0(p+tf(p);g(p)) + dg(p,f(p))\\
=&-\frac{d}{ds}\Big|_{s=0} f(p+sg(p))+ \frac{d}{dt}\Big|_{t=0} g(p) + dg(p,f(p)) \\
=&- df(p,g(p))+ dg(p,f(p))
\end{align*}
\end{proof}

The following Theorem \ref{SatzUeberUKF} comes from \cite[Theorem 2.3]{Gloeckner3}.
\begin{theorem}\label{SatzUeberUKF}
Let $E$ be a locally convex space, $F$ be a Banach space, $P \subseteq E$ and $U\subseteq F$ be open sets and $f\colon P\times U \rightarrow F$ be a $C^r$-map with $r\in \mathbb{N}$. We write $f_p:=f(p,\sbull)\colon U\rightarrow F$ for $p\in P$. Let $p_0\in P$ and $x_0 \in U$ with $f_{p_0}'(x_0) \in \GL(F)$. If $r\geq 2$ or $r=1$ and
\begin{align}
\sup_{(p,x)\in P\times U} \Vert f_{p_0}'(x_0) - f_p'(x)\Vert_{op} < \frac{1}{\Vert f_{p_0}'(x_0)^{-1}\Vert_{op}},
\end{align}
then we find an open $p_0$-neighbourhood $P_0 \subseteq P$ and an open $x_0$-neighbourhood $U_0 \subseteq U$ such that 
\begin{compactenum}[(a)]
\item $f_p(U_0)$ is open in $F$ for all $p \in P$ and $f_p|_{U_0}\colon U_0 \rightarrow f_p(B)$ is a $C^r$-diffeomorphism.
\item $W:=\bigcup_{p\in P_0}(\{p\}\times f_{p}(U_0)$ is open in $E\times F$ and $g \colon W\rightarrow U_0,~ (p,y) \mapsto f_p^{-1}(y)$ is a $C^r$-map.
\item $\Phi \colon P_0 \times U_0 \rightarrow W,~ (p,x) \mapsto (p, f_p(x))$ is a $C^r$-diffeomorphism with inverse $\Psi\colon W \rightarrow P_0 \times U_0,~ (p,z) \rightarrow (p,g(p,z))$.
\end{compactenum}
\end{theorem}

From \cite[Proposition 2.1]{Gloeckner2} we get the following Lemma \ref{cob}.
\begin{lemma}\label{cob}
If $F$ and $H$ are locally convex spaces, $U \subseteq E$ is an open set and $f \colon U \times F \rightarrow H$ is a $C^r$-map that is linear in the second argument, then $f^\vee \colon U \rightarrow \text{L}(F,H)_{c.o.}$ is of class $C^r$ and $f^\vee \colon U \rightarrow \text{L}(F,H)_{b}$ is of class $C^{r-1}$.
\end{lemma}

\begin{lemma}
Let $F$ be a Banach space, $\mathcal{P}_E$ be the set of all continuous seminorms on $E$ and $B_1$ be the closed unit ball in $F$. If we write $\Vert\sbull\Vert_{B,q}$ for a typical seminorm on $\mathcal{L}(F,E)_b$, where $B$ is a bounded set in $F$ and $q\in\mathcal{P}_E$, then the family of seminorms $\big(\Vert\sbull\Vert_{B_1,q}\big)_{q\in\mathcal{P}_E}$ defines the locally convex topology of $\mathcal{L}(F,E)_b$. 
\end{lemma}
\begin{proof}
Obviously the topology that comes from $\big(\Vert\sbull\Vert_{B_1,q}\big)_{q\in\mathcal{P}_E}$ is coarser than the one of $\mathcal{L}(F,E)_b$. To show that it is also finer let $B \subseteq F$ be bounded and $q\in\mathcal{P}_E$. We find $r>0$ with $r\cdot B_1\supseteq B$ and calculate 
\begin{align*}
&\Vert f\Vert_{B,q}\leq \Vert f\Vert_{rB_1,q}=\sup\{q(f(x)):x\in rB_1\}=\sup\{r\cdot q(f(x)):x\in B_1\}\\
=&\Vert f\Vert_{B_1,r\cdot q}.
\end{align*}
\end{proof}

\begin{lemma}\label{iiota}
Let $F$ be a Banach space and $q$ be a continuous seminorm on $E$. For $E_q:=E/q^{-1}(0)$ and $\pi_q \colon E \rightarrow E_q,~ x \mapsto x+q^{-1}(0)$ the map $\iota\colon \mathcal{L}(F,E)_b / (\Vert\sbull\Vert_{B_1,q})^{-1}(0) \hookrightarrow \mathcal{L}\left(F,E_q\right)$, $f+(\Vert\sbull\Vert_{B_1,q}^{-1}(0)) \mapsto \pi_q\circ f$
is well-defined and a topological embedding. Moreover we get for $f\in \mathcal{L}(F,E)$, $g\in\mathcal{L}(F)$ and $\pi_{\Vert\sbull\Vert_{B_1,q}} \colon \mathcal{L}(E,F) \rightarrow \mathcal{L}(E,F)/(\Vert\sbull\Vert_{B_1,q}^{-1}(0)),~f\mapsto f+\Vert\sbull\Vert_{B_1,q}^{-1}(0)$ the equation
\begin{align}\label{GL1000}
\iota\circ\pi_{\Vert\sbull\Vert_{B_1,q}}(f\circ g)=\iota\circ\pi_{\Vert\sbull\Vert_{B_1,q}}(f)\circ g
\end{align}
\end{lemma}
\begin{proof}
Let $f\in\mathcal{L}(F,E)$ with $\Vert f\Vert_{B_1,q}=0$. For $x\in F\setminus\{0\}$, we get 
\begin{align*}
q\circ f(x)=\frac{1}{\Vert x\Vert}\cdot q\circ f\left(\frac{x}{\Vert x\Vert}\right)=0.
\end{align*}
Hence $\iota$ is well-defined. To show that $\iota$ is an isometry  we choose $f\in \mathcal{L}(F,E)$  and calculate
\begin{align*}
&\Vert \pi_q\circ f\Vert_{op} = \sup\{q\circ f(x):x\in B_1\} = \Vert f\Vert_{B_1,q}\\
=&\Vert f+(\Vert\sbull\Vert_{B_1 , q})^{-1}(0)\Vert
\end{align*}
To show (\ref{GL1000}) we calculate 
\begin{align*}
\iota\circ\pi_{\Vert\sbull\Vert_{B_1,q}}(f\circ g)=\pi_q\circ f\circ g=\iota\circ\pi_{\Vert\sbull\Vert_{B_1,q}}(f)\circ g.
\end{align*}
\end{proof}

The following Lemma \ref{Eindeutigkeit1} varies an observation by Gl{\"o}ckner.
\begin{lemma}\label{Eindeutigkeit1}
If $F$ is a Banach space, $\lambda \colon I\rightarrow \mathcal{L}(F)$ is a $C^1$-curve and $\mu_0\in\mathcal{L}(F,E)$, then the initial value problem
\begin{align}
\begin{cases} \label{AWP-TYP1000}
\varphi'(t)&=\varphi(t)\circ\lambda(t)\\
\varphi(0)&=\mu_0
\end{cases}
\end{align}
in $\mathcal{L}(F,E)$ has not more than one solution.
\end{lemma}
\begin{proof}
Let $\varphi_1,\varphi_2\colon]-\varepsilon, \varepsilon[$ be solutions of the initial value problem (\ref{AWP-TYP1000}), $q$ a continuous seminorm of $E$. Moreover let $\pi_{\Vert\sbull\Vert_{B_1,q}}$ and $\iota$ be like in Lemma \ref{iiota}. For $i=1,2$ we define the map $\varphi_{i,q}\colon ]-\varepsilon,\varepsilon[ \rightarrow \mathcal{L}(F,E_q)$, $t\mapsto \iota\circ \pi_{\Vert\sbull\Vert_{B_1 , q}}\circ \varphi_i$ and get
\begin{align*}
&\varphi_{i,p}'(t)=\iota\circ \pi_{\Vert\sbull\Vert_{B_1,q}}(\varphi_i'(t))
=\iota\circ \pi_{\Vert\sbull\Vert_{B_1,q}}(\varphi_i(t)\circ \lambda(t))\\
=&\iota\circ \pi_{\Vert\sbull\Vert_{B_1,q}}(\varphi_i(t))\circ \lambda(t) =\varphi_{i,q}(t)\circ \lambda(t)
\end{align*}
and $\varphi_{i,q}(0)=\pi_q\circ \mu_0$. The composition $\mathcal{L}(F) \times \mathcal{L}(F,E_q) \rightarrow \mathcal{L}(F,E_q),~ (\mu,\psi) \mapsto \psi \circ \mu$ is continuous bilinear. Hence $f\colon I \times \mathcal{L}(F,E_q) \rightarrow \mathcal{L}(F,E_q),~ (t,\psi) \mapsto \psi \circ \lambda(t)$ is Fr{\'e}chet-differentiable. Thus $f$ is continuous and locally Lipschitz-continuous in the second argument. Because $\mathcal{L}(F,E_q)$ is a Banach space, we have $\varphi_{1,q}=\varphi_{2,q}$. Hence $\pi_{\Vert\sbull\Vert_{B_1,q}}\circ \varphi_1=\pi_{\Vert\sbull\Vert_{B_1,q}}\circ \varphi_2$. Because $q$ was an arbitrary continuous seminorm of $E$, we get $\varphi_1=\varphi_2$.
\end{proof}

As mentioned above the following Theorem \ref{BanachBanach} is inspired by \cite[Section 4]{Stefan} and \cite[Theorem 2]{Teichmann} respectively the authors Theorem \cite[Theorem 2.6]{Eyni}.
\begin{theorem}\label{BanachBanach}
Let $F$ be a complemented subspace of $E$ that is a Banach space, $r\geq 2$ and $D$ be an involutiv subbundle of $TM$ with typical fibre $F$. If given $p_0\in M$ there exists an open $p_0$-neighbourhood $U\subseteq M$ and a $C^{r-1}$-vector field $X \colon U\times F \rightarrow TU$  with parameters in $F$ such that
\begin{compactenum}
\item the map $\check{X}\colon F \rightarrow \mathcal{V}(U),~ v\mapsto X(\bl,v)$ is linear,\label{BanachBanach1}
\item we have $\im(X) \subseteq D$,\label{BanachBanach2}
\item the map $X(p_0,\sbull)|^{D_{p_0}} \colon F \rightarrow D_{p_0}$ is an isomorphism of topological vector spaces and\label{BanachBanach3}
\item the vector field $X$ provides a local flow with parameters,\label{BanachBanach4}
\end{compactenum}
then $D$ is a Frobenius distribution. 
\end{theorem}
\begin{proof}
Let $p_0 \in M$ and $\phi$ be a chart around $p_0$ with $\phi(p_0)=0_E$ and $d\phi(D_{p_0})=F$. Because of Lemma \ref{localfb} it is enough to show the statement in the local chart $\phi$. This means we have the following situation: The set $U$ is an open $0$-neighbourhood in $E$. The vector distribution $D\subseteq U\times E$ is a subbundle of $U\times E$ with typical fibre $F$. Hence given $x\in U$ we find a $C^r$-diffeomorphism $\psi\colon U\times E\rightarrow U\times E$ such that $\psi(\{y\}\times E)= \{y\}\times E$, $\pr_2 \circ \psi(y,\sbull)\colon E\rightarrow E$ is an isomorphism of topological vector spaces and $\psi(D)=U\times F$. Given $x\in U$ we write $D_x$ for the subspace ${\pr}_2(D\cap(\{x\}\times E))$ of $E$. By choice of $\phi$ we have $D_{0_E}=F$. 
In abuse of notation we write $X$ of the local representative of $X$ in the chart $\phi$. Hence $X \times F \rightarrow E$ is a $C^r$-map such that
\begin{compactenum}
\item $\check{X} \colon F \rightarrow C^{r-1} (U,E)$ is linear,
\item $X(p,v) \in D_p$ for all $p \in U$ and $v \in F$,\label{imX}
\item $X(0,\sbull)|^{D_0} \colon F \rightarrow D_0=F$ is an isomorphism of topological vector spaces
\item $X$ provides a local flow with parameters. 
\end{compactenum}
We write $\Phi \colon \Omega \rightarrow U$ for the global flow with parameters of $X$. For convenience  we write $X_v:=X(\sbull,v)$, $\Phi^v:=\Phi(\sbull,\sbull,v)$ and $\Omega^v:=\{(t,x)\in \mathbb{R}\times U: (t,x,v) \in \Omega\}$ for $v \in F$ and $t \in \mathbb{R}$.
Since $X\colon U\times F\rightarrow E$ is a $C^1$-map, also $\tilde{X}\colon U\times F\rightarrow F$, $(x,v)\mapsto \psi(x,X(x,v))$ is of class $C^1$. This provides the continuity of $\check{\tilde{X}}\colon U\rightarrow \mathcal{L}(F)$, $x\mapsto \psi(x,\sbull)\circ X(x,\sbull)$, because of Lemma \ref{cob}. Since $X(0_E,\sbull)$ is an isomorphism of topological vector spaces, we assume that $X(x,\sbull)|^{D_x}\colon F \rightarrow D_x$ is an isomorphism of topological vector spaces for all $x\in U$.

Moreover we assume $X(0,\sbull)= \id_F$. We divide the proof in two steps. The first one is to show the following assertion.\\
{\it Assertion I:} Given a vector field $Y\colon U\rightarrow E$, with $Y(x)\in D_x$ for all $x\in U$, also $\big((\Phi_t^{v})^\ast Y\big)(x)\in D_x$ for all $x\in U$, $t \in \mathbb{R}$ and $v\in F$.
Moreover we have 
\begin{align}\label{d2Dy}
d_2\Phi(t,y,v;\sbull)(D_y) = D_{\Phi(t,y,v)}
\end{align}
for $(t,y,v)\in \Omega$.\\
{\it Reason:} For $v \in F$ and $t \in \mathbb{R}$ the vector field $(\Phi_t^v)^\ast Y$ is defined on $\Omega^v_t:=\{x \in U: (t,x,v) \in \Omega\}$. If $\Omega^v_t=\emptyset$ the assertion is clear. Now let $x\in \Omega_t^v$. We have to show  $(\Phi_t^v)^\ast Y(x)\in D_x$.
It exists $w\in F$ with $X(\Phi(t,x,v),w)=Y(\Phi(t,x,v))$. Thus
\begin{align*}
&(\Phi_t^{v})^\ast Y(x) = (d\Phi_t^{v}(x,\sbull))^{-1}\circ Y\circ \Phi_t^{v}(x) =(d\Phi_t^{v}(x,\sbull))^{-1}\circ X_w\circ \Phi_t^{v}(x)\\
=&(\Phi_t^{v})^\ast X_w(x).
\end{align*}
So we only have to show $(\Phi_t^{v})^\ast X_{w}(x)\in D_{x}$. On the interval $I_{v,x}:=\{t\in \mathbb{R}: (t,x,v) \in \Omega\}$ we have 
\begin{align*}
&\frac{\partial}{\partial t}\Big((\Phi_t^{v})^\ast X_{w}(x)\Big)
=\frac{\partial}{\partial s}\Big|_{s=0}\left((\Phi^{v}_{t+s})^{\ast}X_{w}(x)\right)
=\frac{\partial}{\partial s}\Big|_{s=0}\left((\Phi^{v}_{s})^{\ast}(\Phi^{v}_{t})^{\ast}X_{w}(x)\right)\\
=&\left[X_{v}, (\Phi^{v}_{t})^{\ast}X_{w}\right](x) 
=\left[(\Phi^{v}_{t})^{\ast}X_{v}, (\Phi^{v}_{t})^{\ast}X_{w}\right](x)
=(\Phi^{v}_{t})^{\ast}\left[X_{v}, X_{w}\right](x)
\end{align*}
with the help of Lemma \ref{AblLieKlam}, Lemma \ref{ENDL3} and Lemma \ref{ENDL4}.
%
%
Now we define the curve $g_w\colon I_{v,x}\rightarrow E$, $g_w(t):=(\Phi_t^{v})^\ast X_w(x)$ and write $\lambda_x:=X(x,\sbull)|^{D_x}$. Moreover we define $x_t:=\Phi_t^v(x)$. From $[X_v,X_w](x_t) = X(x_t,\lambda_{x_t}^{-1}([X_v,X_w](x_t)))$ we conclude
\begin{align*}
g_w'(t)=(\Phi_t^{v})^\ast [X_v,X_w](x) =(\Phi_t^{v})^\ast X_{\lambda^{-1}_{x_t}([X_v,X_w](x_t))}(x) =g_{\lambda_{x_t}^{-1}([X_v,X_w](x_t))}(t).
\end{align*}
For $t\in I_{v,x}$ we define the maps $A(t)\colon F\rightarrow E$, $u\mapsto g_u(t)$ and $B(t) \colon F\rightarrow F$, $w\mapsto \lambda_{x_t}^{-1}([X_v,X_w](x_t))$. We also define $A\colon I_{v,x} \rightarrow \mathcal{L}(F,E)$, $t\mapsto A(t)$ and $B \colon I_{v,x} \rightarrow \mathcal{L}(F)$, $t \mapsto B(t)$. The curves $A$ and $B$ are of class $C^1$, because $I_{v,x} \times F \rightarrow E,~ (t,w)\mapsto g_w(t)$  and $I_{v,x} \times F \rightarrow F,~ (t,w)\mapsto  \lambda_{x_t}^{-1}([X_v,X_w](x_t))$ are of class $C^{r-1}$ and $r\geq 2$. We get
\begin{align*}
&A'(t).w=\varepsilon_w(A'(t))=d(\varepsilon_w\circ A)(t,1) =\frac{\partial}{\partial t}\big(g_w(t)\big) =g_{\lambda_{x_t}^{-1}([X_v,X_w](x_t))}(t)\\
=&A(t).\lambda_{x_t}^{-1}([X_v,X_w](x_t)) =(A(t)\circ B(t))(w)
\end{align*}
Hence $A$ solves the initial value problem
\begin{align}
\begin{cases}
\varphi'(t)&=\varphi(t)\circ B(t)\\
\varphi(0)&=X(x,\sbull)
\end{cases}
\label{UUU}
\end{align}
in $\mathcal{L}(F,E)$. It exists a solution of the initial value problem (\ref{UUU}) in $\mathcal{L}(F,D_{x})$. From Lemma \ref{Eindeutigkeit1}, we conclude that the image of $A$ lies in $\mathcal{L}(F,D_{x})$. 

It remains to show (\ref{d2Dy}). To this end let $(t,y,v) \in \Omega$ and $f \colon U \rightarrow E$ be a $C^r$-map with $f(p) \in D_p$ for all $p\in U$. We define $x:=\Phi(t,y,v)$ and get $(-t,x,v) \in\Omega$. Hence, $d\Phi^v_t(\Phi_{-t}^v(x),f(\Phi_{-t}^v(x)))\in D_x$. We conclude $d\Phi_t^v(y,f(y))\in D_{\Phi_t^v(y)}$. This shows assertion I, because we can choose any $f$ with the mentioned condition.

Our second aim is to show the following assertion.\\
{\it Assertion II:} Given $(t,y,u)\in \Omega$ we have
\begin{align}
d_3\Phi(t,y,u;\sbull)(F)\subseteq D_{\Phi(t,y,u)}\label{BEHII}.
\end{align}
for the map $d_3\Phi(t,y,u;\sbull)\colon F\rightarrow E$.\\
{\it Reason:} We have 
\begin{align}
d_1\Phi(t,y,u;1)&=X(\Phi(t,y,u),u)\text{ and }\label{BEHII.1}\\
\Phi(0,y,u)&=y.
\end{align}
By differentiating the right-hand side of (\ref{BEHII.1}) in $y$ in direction $h\in E$, we get
\begin{align*}
d_y\big(X(\Phi(t,y,u),u)\big)(y,h)=d_1X\big(\Phi(t,y,u),u;d_2\Phi(t,y,u;h)\big).
\end{align*}
Differentiation of the left-hand side of (\ref{BEHII.1}) in $y$ in direction $h\in E$ leads to 
$\frac{\partial}{\partial s}\frac{\partial}{\partial t}(\Phi(t,y+sh,u))=\frac{\partial}{\partial t}d_2\Phi(t,y,u;h)$.
Here we used the Schwarz-theorem. We conclude
\begin{align}
\frac{\partial}{\partial t}d_2\Phi(t,y,u;h)&=d_1X\big(\Phi(t,y,u),u;d_2\Phi(t,y,u;h)\big)\label{AWP0}\\
d_2\Phi(0,y,u;h)&=h.
\end{align}
Now we differentiate the right-hand side of (\ref{BEHII.1}) in $u$ in direction $h\in F$ and get
\begin{align*}
&d_u\big(X(\Phi(t,y,u),u)\big)(u,h)=dX\big((\Phi(t,y,u),u);(d_3\Phi(t,y,u;h),h)\big)\\
=&d_2X\big(\Phi(t,y,u),u;h\big)+d_1X\big(\Phi(t,y,u),u;d_3\Phi(t,y,u;h)\big)\\
=&X\big(\Phi(t,y,u),h\big)+d_1X\big(\Phi(t,y,u),u;d_3\Phi(t,y,u;h)\big)
\end{align*}
Differentiation of the left-hand side of (\ref{BEHII.1}) leads to 
\begin{align*}
\frac{\partial}{\partial s}\frac{\partial}{\partial t}(\Phi(t,y,u+sh))=\frac{\partial}{\partial t}d_3\Phi(t,y,u;h).
\end{align*}
Hence we get
\begin{align*}
\frac{\partial}{\partial t}d_3\Phi(t,y,u;h)&=X\big(\Phi(t,y,u),h\big)+d_1X\big(\Phi(t,y,u),u;d_3\Phi(t,y,u;h)\big)\text{ and}\\
d_3\Phi(0,y,u;h)&=0.
\end{align*}
Thus $t\mapsto d_3\Phi(t,y,u;\sbull)$ solves the initial value problem
\begin{align}
\begin{cases}\label{AWP}
\sigma'(t)&=X\big(\Phi(t,y,u),\sbull\big)+d_1X\big(\Phi(t,y,u),u;\sbull\big)\circ \sigma(t)\\
\sigma(0)&=0
\end{cases}
\end{align}
in $\mathcal{L}(F,E)$.

Because $\Phi(t,\sbull,u)$ is a diffeomorphism, we get $d_2\Phi(t,y,u;\sbull)\in\GL(E)$. With $\Phi(t,\Phi(-t,y,u),u)=y$ we get
\begin{align}\label{BanInv}
\big(d_2\Phi(t,y,u;\sbull)\big)^{-1}=d_2\Phi(-t,\Phi(t,y,u),u;\sbull).
\end{align} 
The map $f\colon I\times F\rightarrow E$, $(t,v)\mapsto d_2\Phi(-t,\Phi(t,y,u),u;X(\Phi(t,y,u),v))$ is of class $C^1$ and $\int_0^tf(s,v)ds = t\cdot \int_0^1 f(ts,v)ds$.
Thus $f_1\colon I\times F \rightarrow E$, $(t,v) \mapsto \int_0^tf(s,v)ds$ is of class $C^1$. We conclude that $f_2\colon I\times F \rightarrow E$, $(t,v) \mapsto d_2\Phi(t,y,u;\int_0^tf(s,v)ds)$ is of class $C^1$. Hence 
$\eta:=\widecheck{f_2}\colon I\rightarrow \mathcal{L}(F,E)_c\subseteq C(F,E)_{c.o.}$
is a $C^1$-map. We want to show that $\eta$ is a solution of the initial value problem (\ref{AWP}). Given $v\in F$ the evaluation map $\varepsilon\colon \mathcal{L}(F,E)\rightarrow E$, $\lambda\mapsto\lambda(v)$ is continuous linear. Therefore we only have to show that for all $v\in F$ the curve $\tau\colon I \rightarrow E$, $t\mapsto d_2\Phi(t,y,u;\int_0^tf(s,v)ds)$ is a solution of the initial value problem 
\begin{align}
\begin{cases}\label{AWP2}
\frac{d}{dt}\sigma(t)&=d_1X(\Phi(t,y,u),u;\sigma(t))+X(\Phi(t,y,u),v)\\
\sigma(0)&=0,
\end{cases}
\end{align}
where $\sigma$ is a curve in $E$. We define the map $H\colon I\times E \rightarrow E$, $(t,w) \mapsto d_2 \Phi(t,y,u;w)$ and get 
\begin{align*}
&\tau'(t)=\frac{\partial}{\partial t}(H\circ ({\id}_I(t),f_1(t,v))) = d\big(H\circ ({\id}_I,f_1(\sbull,v))\big)(t,1)\\
=&dH( (t,f_1(t,v));(1,f(t,v)))\\
=&d_1H\big(t,f_1(t,v);1\big)+d_2H\big(t,f_1(t,v);f(t,v)\big).
\end{align*}
On the one hand we have 
\begin{align}
&d_2H(t,f_1(t,v);f(t,v)) = H(t,f(t,v))\nonumber\\
=& d_2\Phi\big(t,y,u;\underbracket{d_2\Phi(-t,\Phi(t,y,u),u;X(\Phi(s,y,u),v))}_{=(d_2\Phi(t,y,u;\sbull))^{-1}(X(\Phi(s,y,u),v))}\big) = X(\Phi(s,y,u),v)\label{TEIL1}
\end{align}
and on the other 
\begin{align}
&d_1H(t,f_1(t,v);1) = \frac{\partial}{\partial h}\Big(d_2\Phi(h,y,u;f_1(t,v))\Big)\Big|_{h=t}\nonumber\\
=&\frac{\partial}{\partial h'}\Big(\frac{\partial}{\partial h}\Big(\Phi(h,y+h'\cdot f_1(t,v),u)\Big)\Big|_{h=t}\Big)\Big|_{h'=0}\nonumber\\
=&\frac{\partial}{\partial h'}\Big(X(\Phi(t,y+h'\cdot f_1(t,v),u),u)\Big)\Big|_{h'=0}\nonumber\\
=&d_1X(\Phi(t,y,u),u;d_2\Phi(t,y,u;f_1(t,v)))
=d_1X(\Phi(t,y,u),u;\tau(t)).\label{TEIL2}
\end{align}
Thus $\tau$ is a solution of (\ref{AWP2}) and so $\eta$ solves the initial value problem (\ref{AWP}).

Now we show that the solution of (\ref{AWP}) is unique. It is enough to show that for every $h\in F$ the initial value problem 
\begin{align*}
\begin{cases}
g'(t)&=X(\Phi(t,y,u),h)+d_1X(\Phi(t,y,u),u;g(t))\\
g(0)&=0,
\end{cases}
\end{align*}
where $g$ is a curve in $E$, has a unique solution. Obviously it is sufficient to show that the initial value problem
\begin{align}
\begin{cases}
g'(t)&=d_1X(\Phi(t,y,u),u;g(t))\\
g(0)&=0
\end{cases}
\label{AWPinE}
\end{align}
has at most one solution.
%
%
We define $\tilde{\Omega}:=\{(t,y) \in \mathbb{R}\times U: (t,y,u) \in \Omega\}$ and consider the map $\tilde{\Phi}\colon \tilde{\Omega} \times E \rightarrow U\times E$, $(t,y,w) \mapsto T\Phi_t^{u}(y,w)$ which is a local $C^r$-action on $U\times E$, because of the chain rule of tangential-maps. The vector field
$\tilde{X}\colon U\times E \rightarrow E\times E$, $(y,w) \mapsto(X(y,u),d_1X(y,u;w))$
has the local flow $\tilde{\Phi}$, because with (\ref{AWP0}) we get
\begin{align*}
&\frac{d}{dt}\big(\tilde{\Phi}(t,y,w)\big)\Big|_{t=0} = \frac{d}{dt}\big(\Phi_t^{u}(y),d_2\Phi(t,y,u;w)\big)\Big|_{t=0}\\
=&\big(X(y,u),d_1X(y,u;w)\big).
\end{align*}
Now let $g_1$ and $g_2$ be solutions of (\ref{AWPinE}), defined on the interval $I$. For $i=1,2$ the curve $G_i\colon I \rightarrow U\times E$, $t \mapsto (\Phi(t,y,u),g_i(t))$ is a solution of the initial value problem 
\begin{align*}
\begin{cases}
G_i'(t)=\tilde{X}(G_i(t))\\
G_i(0)=(y,0),
\end{cases}
\end{align*}
because 
\begin{align*}
G_i'(t)=\big(X(\Phi(t,y,u),u) , d_1X(\Phi(t,y,u),u;g_i(t))\big)=\tilde{X}(G_i(t)).
\end{align*}
Hence $G_1(t)=G_2(t)$ and so $g_1(t)=g_2(t)$. We get the uniqueness-statement.\\
Now we conclude
\begin{align}\label{intdar}
d_3\Phi(t,y,u;\sbull)=\eta(t)= \left(v\mapsto d_2\Phi\left(t,y,u;\int_0^tf(s,v)ds\right)\right).
\end{align}
With (\ref{d2Dy}) we get $D_y=\big(d_2\Phi(s,y,u;\sbull)\big)^{-1}(D_{\Phi(s,y,u)})$. From the condition (\ref{imX}) we get $X(\Phi(s,y,u),\sbull).F\subseteq D_{\Phi(s,y,u)}$. With (\ref{BanInv}) we conclude $\Image(f)\subseteq D_y$. Hence with (\ref{intdar}) we get $d_3\Phi(t,y,u;\sbull)(F)\subseteq d_2\Phi(t,y,u;\sbull).D_y=D_{\Phi(s,y,u)}$.

Now we construct a Frobenius chart around $0_E$. To this end let $\tilde{F}$ be the topological vector complement of $F$ in $E$. We choose open $0$-neighbourhoods $V^{(1)} \subseteq \tilde{F}$, $V^{(2)} \subseteq F$ and a symmetric interval $I \subseteq \mathbb{R}$ such that $V:=V^{(1)}\times V^{(2)} \subseteq U$  and $I\times V \times V^{(2)} \subseteq \Omega$.

We have $\frac{\partial}{\partial s} \Phi(s,0,0) =X(\Phi(t,0,0) ,0)=0$ and $\Phi(0,0,0)=0$. Hence $\Phi(t,0,0)=0$ for all $t \in I$.

We saw $\Image(f)\subseteq D_y$ in the calculation above. Taking $y=0_E$ and $u=0_F$ we get $d_2\Phi(-s,0_E,0_F;v)\in F$ for $s \in I$ and $v \in F$. We define the map $\lambda \colon I\times F \rightarrow F$, $(s,v) \mapsto d_2\Phi(-s,0_E,0_F;v)$. Because $\check{\lambda} \colon I \rightarrow \mathcal{L}(F)$, $s \mapsto d_2\Phi(-s,0,0,\sbull)$ is continuous, the map $\tilde{\lambda} \colon I\times I \rightarrow \mathcal{L}(F)$, $(t,s) \mapsto t\cdot \check{\lambda}(s)$ is continuous and because $\tilde{\lambda}(0,0) = \id_F$ we can choose $0<t<1$ with $\Vert t\cdot\lambda(s,\sbull)-\id_F\Vert_{op}<\frac{1}{2}$ for all $ s\in [0,t]$. We show 
\begin{align}
d_3\Phi(t,0_E,0_F;\sbull) \in \mathcal{L}(F)^\ast.\label{L(F)Stern}
\end{align}
With (\ref{intdar}) we get 
$d_3\Phi(t,0,0;\sbull) = v \mapsto d_2\Phi(t,0_E,0_F; \int_0^t \lambda(s,v)ds)$.
The map $d_2\Phi(t,0,0;\sbull) \colon E \rightarrow E$ is an isomorphism of topological vector spaces and so we see $d_2\Phi(t,0,0;\sbull)|_F^F \in \mathcal{L}(F)^\ast$ with Assertion I. Hence it remains to show that the map $\mu \colon F\rightarrow F, v \mapsto \int_0^t \lambda(s,v)ds$ is an isomorphism. To show $\Vert \mu-\id_F\Vert_{op}<1$ we choose $v \in F$ and calculate 
\begin{align*}
&\left\Vert \int_0^t \lambda(s,v) ds-v \right\Vert_F =  \left\Vert \int_0^t \lambda(s,v) -\frac{v}{t} ds \right\Vert_F \leq  \int_0^t \left\Vert \lambda(s,v) -\frac{v}{t} \right\Vert_F ds\\
=&\frac{1}{t} \int_0^t \Vert t\lambda(s,v) -v \Vert_F ds \leq \frac{\Vert v\Vert_F}{t} \int_0^t \underbracket{\Vert t\lambda(s,\sbull) -\id_F \Vert_{op}}_{<\frac{1}{2}}  ds \leq \frac{\Vert v\Vert_F}{t} \cdot t \cdot \frac{1}{2} = \frac{\Vert v\Vert_F}{2}.
\end{align*}

Now we show that $\zeta \colon V^{(1)} \times V^{(2)}  \rightarrow E$, $(x,w)\mapsto \Phi(-t,x,w)$
has open image and is a diffeomorphism onto its image. To this end we consider the $C^r$-map
$b\colon V^{(1)}\times V^{(2)}\times V^{(2)} \rightarrow \tilde{F}\times F, ~ (z,w,v) \mapsto \Phi(t,(z,v),w)$.
We have $\Phi(t,0,0)=0$. With Assertion II we get $d_3\Phi(t,0,0;\sbull).F \subseteq D_{\Phi(t,0,0)} =F$ and with (\ref{L(F)Stern}) we conclude 
\begin{align*}
d_2b_{2}(0,0,0;\sbull)={\pr}_2(d_3\Phi(t,0_E,0_F;\sbull)) = d_3\Phi(t,0_E,0_F;\sbull)\in \mathcal{L}(F)^\ast.
\end{align*}

With Theorem \ref{SatzUeberUKF} we get that after shrinking $V^{(1)}$ and $V^{(2)}$ the map $b_{2}(z,\sbull,v)\colon V^{(2)} \rightarrow F$ has open image and is a diffeomorphism onto its image. Moreover we get that  
$\Psi\colon V^{(1)}\times V^{(2)}\times V^{(2)}  \rightarrow E\times F ~ (z,w,v)\mapsto ( (z,v),b_{2}(z,w,v))$
has open image and is a diffeomorphism onto its image. We have $\Phi(t,0,0)=0$ and so $\Psi(0,0,0)=(0,0)$. We choose $0$-neighbourhoods $W^{(1)} \subseteq V^{(1)} \subseteq \tilde{F}$ and $W^{(2)} \subseteq V^{(2)}\subseteq F$ such that $W^{(1)}\times W^{(2)} \times W^{(2)} \subseteq \im (\Psi)$. Hence $\Psi^{-1}(z,v,0)=(z,b_{2}(z,\sbull,v)^{-1}(0),v)$ for $(z,v)\in W^{(1)} \times W^{(2)}$. We define $W:=W^{(1)}\times W^{(2)}$.

For the map $u\colon W^{(1)}\times W^{(2)} \rightarrow V^{(2)}~ (z,v) \mapsto (\Psi^{-1})_{2}(z,v,0)$ we get $b_{2}(z,u(z,v),v)=0$, because of $(\Psi^{-1})_{2}(z,v,0) = b_{2}(z,\sbull,v)^{-1}(0)$. We define the map $\xi\colon W^{(1)}\times W^{(2)} \rightarrow E~ (z,v) \mapsto (b_{1}(z,u(z,v),v),u(z,v))$.  In the following we show that $\xi|_{\xi^{-1}(V)}$ is inverse to $\zeta|_{\zeta^{-1}(W)}$. To this end we calculate 
\begin{align*}
&\zeta\circ \xi(z,v) = \zeta(b_{1}(z,u(z,v),v),u(z,v)) = \Phi(-t,\underbracket{b_{1}(z,u(z,v),v)}_{=b(z,u(z,v),v)},u(z,v))\\
=& \Phi(-t,\Phi(t,(z,v),u(z,v)),u(z,v)) = (z,v)
\end{align*}
Given $(x,w)\in \zeta^{-1}(W)$ we have
\begin{align}\label{BaDi-x}
b(\zeta_{1}(x,w),w,\zeta_{2}(x,w)) = \Phi(t,\zeta(x,w),w) = x.
\end{align}
And so we get $b_{2}(\zeta_{1}(x,w),w,\zeta_{2}(x,w))=0$ respectively $u(\zeta(x,w))=w$. Hence
\begin{align*}
&\xi\circ \zeta (x,w) =(b_{1}(\zeta_{1}(x,w),u(\zeta(x,w)),\zeta_{2}(x,w)),u(\zeta(x,w)))\\
=&(b_{1}(\zeta_{1}(x,w),w,\zeta_{2}(x,w)),w) = (x,w).
\end{align*}

We define $U_\varphi:= \xi^{-1}(V)$, $V_\varphi:=\zeta^{-1}(W)$ and $\varphi:=\xi|_{U_\varphi}^{V_\varphi}$. In particular we get $\varphi^{-1}=\zeta|_{V_\varphi}$. After shrinking $V_\varphi$ we assume $V_\varphi = V^{(1)}_{\varphi} \times V^{(2)}_{\varphi}$ with $V^{(1)}_{\varphi} \subseteq V^{(1)}$ and $V^{(2)}_{\varphi} \subseteq V^{(2)}$. We show that $\varphi$ is a Frobenius  chart around $0$.
%
%
%
It is sufficient to show $d\varphi(\{p\}\times D_p) = F$ respectively $((d\varphi)(p,\sbull))^{-1}(F)=D_p$ for all $p \in U_\varphi$, because of Remark \ref{FB-charakter}. This is equivalent to show $d\varphi^{-1}(x,w;\sbull)(F)= D_{\varphi^{-1}(x,w)}$ respectively $d_2\zeta(x,w;\sbull)(F)= D_{\zeta(x,w)}$ for all $(x,w) \in V_\varphi = V_\varphi^{(1)}\times V_\varphi^{(2)}$.

Since (\ref{BEHII}) the map $\lambda\colon V_\varphi^{(1)} \times V_\varphi^{(2)} \rightarrow \mathcal{L}(F)~ (x,w) \mapsto \psi_{2}(x,d_3\Phi(t,x,w;\sbull))$ is well-defined and continuous. Because of $\lambda(0_{\tilde{F}} , 0_F)\in\mathcal{L}(F)^\ast$ we assume $\lambda(x,w)\in\mathcal{L}(F)^\ast$ for all $(x,w)\in V_\varphi^{(1)}\times V_\varphi^{(2)}$. Hence $d_3\Phi(t,x,w;\sbull)=\psi_{2}(x,\sbull)^{-1}\circ \lambda(x,w)\in\mathcal{L}(F,D_x)$
is an isomorphism of topological vector spaces for $x\in V_\varphi^{(1)}$.
\end{proof}

\section{Applications of Frobenius theorems in infinite-dimensional Lie-theory}\label{LieTheory}

In \cite[Problem VI.4]{Neeb} Neeb stated the following open problem.
\begin{problem2}
Let $G$ be a regular Lie group. Is every closed Lie subalgebra $\mathfrak{h}\subseteq L(G)$ with finite codimension integrable?  
\end{problem2}
In \cite{Gloeckner} Gl{\"o}ckner proposed that generalisations of the Frobenius theorems can be used to solve this problem. In \cite{Eyni} we used Frobenius theorems for Co-Banach Distributions to solve this problem. Moreover it was possible to solve the problem for a finite-dimensional Lie subalgebra. In this section we want consider the case of a Banach Lie algebra. The main work was done in Section \ref{Frobenius theorem}. We simply use the methods from \cite[Chapter 4]{Eyni}. The only new thing is Theorem \ref{satzBanach} and its proof at the end of this section. For the convenience of reader we present the definitions, results and proofs of \cite[Chapter 4]{Eyni}. We also mention that in \cite[Chapter 4]{Eyni} we used ideas of \cite[VI, \S5. Lie Groups and Subgroups]{Lang}, where Lang discussed the case of Banach Lie groups. Besides the more general approach we corrected in \cite[Chapter 4]{Eyni} an inaccuracy of Lang in \cite[Chapter VI. Lemma 5.3]{Lang} with the help Gl{\"o}ckner.

First we state the main result of this section and will prove it at the end.

\begin{theorem}\label{satzBanach}
Let $G$ be a Lie group modelled over a locally convex space and $\mathfrak{h}\subseteq L(G)$ be a Lie subalgebra that is complemented as a topological subspace and is a Banach space. If $G$ provides an exponential map, then we can find a Lie group $H$ that is a subgroup of $G$ and an immersed submanifold of $G$ such that $L(H)=\mathfrak{h}$.
\end{theorem}

\begin{definition}\label{links}
Let $G$ be a Lie group, $U\subseteq G$ be open and $X\colon U\rightarrow TG$ be a vector field of $G$. Given $g_0\in U$ we define the left invariant vector field $X^{l}_{g_{0}}:U \rightarrow TU\subseteq TG,~ g \mapsto T\lambda_{g}(w)$ with $w:=T\lambda_{g_0^{-1}}(X(g_0))$. Obviously $X_{g_0}^l(g_0)=X(g_0)$.
\end{definition}

\begin{remark}
If in the situation of Definition \ref{links} $\mathfrak{h}\subseteq L(G)$ is a closed Lie subalgebra and $X$ is a vector field with its image lying in $D:=\bigcup_{g\in G}T\lambda_g(\mathfrak{h})$, then also the image of $X_{g_0}^l$ lies in $D$.
\end{remark}

\begin{lemma}\label{leftinv}
If $G$ is a Lie group, $\mathfrak{h} \subseteq L(G)$ is a closed Lie subalgebra, $U\subseteq G$ is open and $X_1, X_2 \colon U \rightarrow TG$ are left invariant vector fields with images in $D:=\bigcup_{g\in G}T\lambda_{g}(\mathfrak{h})$, then also the image of $[X_1,X_2]$ lies in $D$.
\end{lemma}
\begin{proof}
The vector field $[X_1,X_2]$ is left invariant. Hence
\begin{align*}
[X_{1},X_{2}](g) =T\lambda_{g}([X_{1},X_{2}](1)) = T\lambda_{g}(\underbracket{[X_{1}(1),X_{2}(1)]}_{\in\mathfrak{h}})\in D.
\end{align*}
for $g \in U$.
\end{proof}

The following Lemma \ref{Anwendung1} was discussed in \cite[Chapter VI. Lemma 5.3]{Lang} for Banach Lie groups. But Lang's proof for the involutivity of the considered vector distribution is not quite correct. In general it is not possible to reduce the problem to the local case by the use of local charts, in the way Lang did it, because in general there exists no local trivialisation of the tangent bundle that is simultaneously induced by a chart of the Lie group and a sub bundle trivialisation of the distribution. As mentioned above our proof Lemma \ref{Anwendung1} was stated in \cite[Chapter 4]{Eyni}. The inaccuracy was noticed by the author and with the help of Gl{\"o}ckner  it was possible to find an other proof.
\begin{lemma}\label{Anwendung1}
Given a Lie group $G$ and a closed Lie subalgebra $\frak{h}\subseteq L(G)$, the vector distribution $D:=\bigcup_{g\in G}T\lambda_{g}(\frak{h})$ is an involutiv subbundle of $TG$ with typical fibre $\frak{h}$, if we identify the modelling space of $G$ with $L(G)$. 
\end{lemma}
\begin{proof}
First we show that $D$ is a subbundle of $TM$. We define $D_{g}:=T\lambda_{g}(\frak{h})$ for $g \in G$ and $\psi \colon TG \rightarrow G\times L(G)$, $v \mapsto (\pi (v), \pi(v)^{-1}.v)$. Given $h \in \frak{h}$ and $g\in G$ we get $\pi(T\lambda_g(h))=g$ and $\pi(T\lambda_g(h))^{-1}.T\lambda_g(h) = g^{-1}.T\lambda_g(h)=h$. Hence $\psi (D)=G\times \frak{h}$ and $\psi(T\lambda_g(h)) =(g,h)$.
It remains to show that $D$ is involutiv. To this end let $X_{1},X_{2}\colon G\rightarrow TG$ be vector fields of $G$ with $\Image(X_{1}),\Image(X_{2})\subseteq D$ and $g_{0}\in G$. 
We find a chart $\phi\colon U\rightarrow V \subseteq L(G)$ of $G$ around $g_{0}$ such that $d\phi(D_{g_{0}})=\frak{h}$. The map
$\theta\colon TU\rightarrow U\times E,~ \theta=(\phi^{-1}\times {\id}_{E})\circ T\phi$
is a local trivialisation of $TG$ around $g_{0}$. By abuse of notation we write $\psi$ for the restriction of $\psi$ to $TU$. We define the map $A\colon V\times E\rightarrow E,~ (x,y)\mapsto g_{\psi,\theta}(\phi^{-1}(x),y)$ with $g_{\psi,\theta}:={\pr}_{2}\circ \theta\circ\psi^{-1}$. And write $A(x):=A(x,\sbull)$ for $x\in V$. The diagramme 
\begin{align*}
\begin{xy}
\xymatrix{
TU\ar[r]^-{\psi}\ar[rd]_-\theta & U\times E\ar[d]^-{(g,y)\mapsto(g,A(\phi(g),y))}\\
&U\times E
}
\end{xy}
\end{align*}
commutes. Let $X_{1,\phi}$ and $X_{2,\phi}$ be the local representatives of the vector fields $X_{1}$ and $X_{2}$ in the chart $\phi$. We define $X_{1,\psi}:={\pr}_{2}\circ\psi\circ X_{1}\circ \phi^{-1}$ and $X_{2,\psi}:={\pr}_{2}\circ\psi\circ X_{2}\circ \phi^{-1}$. Let $x\in V$. For $i=1,2$ we get 
\begin{align}
&A(x).(X_{i,\psi}(x))=A(x).(\psi_{2}\circ X_{i}\circ\phi^{-1}(x))  =A(x).(\psi_{2}(X_{i}(\phi^{-1}(x))))\nonumber\\ 
=&d\phi(X_{i}\circ\phi^{-1}(x)) =X_{i,\phi}(x)\label{Lie-A(x)}.
\end{align}
If we show
\begin{align}\label{FERtiG}
[X_{1,\phi},X_{2,\phi}](x)\in A(x).\frak{h},
\end{align}
we get 
\begin{align*}
&T\phi^{-1}({\id},[X_{1,\phi},X_{2,\phi}])(x)\in T\phi^{-1}(x,A(x).\frak{h}) \subseteq \psi^{-1}(U\times \frak{h})=D\cap TU\\
\Rightarrow&  [X_{1},X_{2}](p)\in D ~~:\forall p\in U.
\end{align*}
Hence it remains to show (\ref{FERtiG}). Equation (\ref{Lie-A(x)}) leads to
\begin{align*}
&dX_{2,\phi}(x,X_{1,\phi}(x))=d(x\mapsto A(x,X_{2,\psi}(x)))(x,X_{1,\phi}(x))\\
=&dA\big((x,X_{2,\psi}(x));(X_{1,\phi}(x),dX_{2,\psi}(x,X_{1,\phi}(x)))\big)\\
=&d_{1}A\big(x,X_{2,\psi}(x);X_{1,\phi}(x)\big)+d_{2}A\big(x,X_{2,\psi}(x);dX_{2,\psi}(x,X_{1,\phi}(x))\big)\\
=&d_{1}A\big(x,X_{2,\psi}(x);X_{1,\phi}(x)\big)+A(x).dX_{2,\psi}(x,X_{1,\phi}(x)).
\end{align*}
With an analogous calculation for $dX_{2,\phi}(x,X_{1,\phi}(x))$ we get
\begin{align}
&[X_{1,\phi},X_{2,\phi}](x)=A(x).(\underbracket{dX_{2,\psi}(x,X_{1,\phi}(x))-dX_{1,\psi}(x,X_{2,\phi}(x))}_{\in \frak{h}})\nonumber\\
+&\underbracket{d_{1}A\big(x,X_{2,\psi}(x) ; X_{1,\phi}(x)\big) - d_{1}A\big(x,X_{1,\psi}(x);X_{2,\phi}(x)\big)}_{=:(\ast)}\label{Anwendung4}
\end{align}
It remains to show $(\ast)\in A(x).\frak{h}$. Writing $h:=\phi^{-1}(x)$, we define the left invariant vector fields $X_{1,h}^{l}$ and $X_{2,h}^{l}$ like in Definition \ref{links}. The images of $X_{1,h}^{l}$ and $X_{2,h}^{l}$ lie in $D$. Hence $[X_{1,h}^{l},X_{2,h}^{l}](x)\in D$, because of Lemma \ref{leftinv} and therefore $[X^{l}_{1,h,\phi},X^{l}_{2,h,\phi}](x)\in A(x).\frak{h}$, where $X^{l}_{1,h,\phi}$ and $X^{l}_{2,h,\phi}$ are the local representatives of the vector fields $X^{l}_{1,h}$ and $X^{l}_{2,h}$. We can repeat the calculation that led us to equation (\ref{Anwendung4}), also for $X_{1,h}^{l}$ and $X_{2,h}^{l}$ and get 
\begin{align*}
&\underbracket{[X^{l}_{1,h,\phi},X^{l}_{2,h,\phi}](x)}_{\in A(x).\frak{h}}=A(x).(\underbracket{dX^{l}_{2,h,\psi}(x,X^{l}_{1,h,\phi}(x))-dX^{l}_{1,h,\psi}(x,X^{l}_{2,h,\phi}(x))}_{\in \frak{h}})\\
+&d_{1}A\big(x,X^{l}_{2,h,\psi}(x);X^{l}_{1,h,\phi}(x)\big)-d_{1}A\big(x,X^{l}_{1,h,\psi}(x);X^{l}_{2,h,\phi}(x)\big).
\end{align*}
Thus
\begin{align*}
d_{1}A\big(x,X^{l}_{2,h,\psi}(x);X^{l}_{1,h,\phi}(x)\big)-d_{1}A\big(x,X^{l}_{1,h,\psi}(x);X^{l}_{2,h,\phi}(x)\big)\in A(x).\frak{h}
\end{align*}
and because $X^{l}_{i,h,\psi}(x)=X_{i,\psi}(x)$ for $i=1,2$, we get $(\ast)\in A(x).\frak{h}$. Hence (\ref{FERtiG}) is shown and we are done.
\end{proof}


As mentioned above our proof the following Lemma \ref{Anwendung6} was stated in \cite[Chapter 4]{Eyni} and there we followed the ideas of \cite[Chapter VI. Theorem 5.4]{Lang}, but Lang did not discuss whether the constructed group $H$ is actually a Lie group.
\begin{lemma}\label{Anwendung6}
If the vector bundle $D$ in Lemma \ref{Anwendung1} is a Frobenius distribution, then we find a Lie group $H$ that is an integral manifold for $D$ and a  subgroup of $G$.
\end{lemma}
\begin{proof}
Because $D$ is a Frobenius distribution, we find a maximal integral manifold $H$ for $D$ that includes $1$. First we show that $H$ is a subgroup of $G$. To this end let $h\in H$. We show $h^{-1}\in H$. The set $h^{-1}\cdot H$ is an integral manifold for $D$, because for $g\in H$ we have
\begin{align}
&T_{h^{-1}g}(h^{-1}H)=T_{h^{-1}g}(\lambda_{h^{-1}}(H))=T\lambda_{h^{-1}}(T_{g}H)\nonumber\\
=&T\lambda_{h^{-1}}(D_{g})=T\lambda_{h^{-1}g}\mathfrak{h}=D_{h^{-1}g}.\label{Anwendung5}
\end{align}
Moreover $1\in h^{-1}H$. By the maximality of $H$ we get $h^{-1}H \subseteq H$. We have $h^{-1}\cdot 1=h^{-1}\in H$, because $1\in H$. 

Now we show $HH\subseteq H$. Let $h\in H$. A similar calculation as in (\ref{Anwendung5}) shows that $hH$ is an integral manifold for $D$ and contains $1$, because $h^{-1}\in H$. Hence $h\cdot H=H$. And thus $HH=H$. We conclude that $H$ is a subgroup of $G$.

To show that $H$ is a Lie group, we have to show that the multiplication and inversion on $H$ are smooth. This follows not directly from the fact that $H$ is a subgroup of the Lie group $G$, because $H$ is no submanifold, but an immersed submanifold of $G$. For $g\in G$ we choose a Frobenius chart $\phi_{g}\colon U_{g}\rightarrow W_{g,1}\times W_{g,2}$ around $g$ with $\phi_{g}(g)=0$. The submanifold 
$S_{g}=\{\phi_{g}^{-1}(x,0):x\in W_{g,1}\}$
is an integral manifold of $D$ and every integral manifold that is contained in $U_{g}$ and intersects $S_{g}$ is a subset of $S_{g}$. First we show that the left multiplication $\lambda_{h}\colon H\rightarrow H$ is smooth for all $h\in H$. Let $g\in G$. 
We write $((h\cdot S_{g})\cap U_{hg})_{hg}$ for the component of $hg$ in $(h\cdot S_{g})\cap U_{hg}$. Obviously $((h\cdot S_{g})\cap U_{hg})_{hg}$ is an integral manifold for $D$ and contained in $U_{hg}$.
Hence $\lambda_{h}((S_{g}\cap \lambda_{h}^{-1}(U_{hg}))_{g})\subseteq S_{hg}$. Because $\lambda_{h}|_{(S_{g}\cap \lambda_{h}^{-1}(U_{hg}))_g}$ is smooth as a map to $G$ and its image is contained in the submanifold $S_{hg}$ and $S_{hg}$ is open in $H$, we get that $\lambda_{h}|_{(S_{g}\cap \lambda_{h}^{-1}(U_{hg}))_g}$ is smooth as a map to $H$. Thus $\lambda_{h}\colon H\rightarrow H$ is a smooth map. Hence $\lambda_{h}\colon H\rightarrow H$ is a diffeomorphism. 

Let $V\subseteq U_{1}$ be an open, symmetric $1$-neighbourhood with $VV\subseteq U_{1}$. 
Given $g\in (S_{1}\cap V)_1$ the set $g^{-1}((S_{1}\cap V)_1)$ is an integral manifold for $D$ that contains $1$ and is a subset of $U_{1}$. Thus $g^{-1}((S_{1}\cap V)_1) \subseteq S_{1}$. The map 
$\theta\colon (S_{1}\cap V)_1\times (S_{1}\cap V)_1 \rightarrow G,~ (g,h) \mapsto g^{-1}h$
is smooth. Since $\Image(\theta)\subseteq S_{1}$ and $S_{1}$ is a submanifold of $G$, we get that $\theta|^{S_{1}}$ is smooth. We write $i_{H}$ for the inversion on $H$ and $m_{H}$ for the  multiplication on $H$. The sets $S_{1}\cap V$ and $(S_{1}\cap V)_1$ are open in $H$ and $i_{H}|_{(S_{1}\cap V)_1}=\theta(\sbull,1)\colon (S_{1}\cap V)_1 \rightarrow S_{1},~ g\mapsto g^{-1}$ is smooth, because $S_{1}$ is open in $H$. Therefore $m_{H}\colon(S_{1}\cap V)_1 \times (S_{1}\cap V)_1 \rightarrow S_{1},~ (g,h)\mapsto g\cdot h=\theta(g^{-1},h)=\theta(i_{H}(g),h)$ is smooth. 

We define $\tilde{V}:=(S_{1}\cap V)_1$ and show that $<\tilde{V}>$ is open and closed in $H$ and so $<(S_{1}\cap V)_1>=H$, because $H$ is connected. Given $g\in <\tilde{V}>$ we have $g\in g\cdot \tilde{V}\subseteq <\tilde{V}>$. And $g\cdot \tilde{V}$ is open in $H$, because the left multiplication is a homeomorphism of $H$ and $\tilde{V}$ is open in $H$. For $g\notin \tilde{V}$ we get $(g\cdot \tilde{V})\cap <\tilde{V}>=\emptyset$. Hence $H_{1}\setminus <\tilde{V}>$ is open. 

With the theorem over the local description of Lie groups we find a unique manifold structure on $H$ such that $H$ is a Lie group and $(S_{1}\cap V)_1$ is an open submanifold in $H$. We write $H_{2}$ for the set $H$ equipped with this Lie group structure and $H_{1}$ for the set $H$ equipped with the original manifold structure that made $H$ to an integral manifold of $D$. Now we choose a chart $\psi\colon U_{\psi}\rightarrow V_{\psi}$ of $(S_{1}\cap V)_1$. Given $g\in H$ the map $\psi_{g}\colon gU_{\psi} \rightarrow V_{\psi},~ h\mapsto \phi(\lambda_{g^{-1}}(h))$ is a chart of $H_{2}$ around $g$. Because the left multiplication on $H_{1}$ is a diffeomorphism, $\psi_{g}$ is a chart of $H_{1}$ around $g$. Hence $H_{1}$ and $H_{2}$ have a common atlas and must therefore coincide.
%
\end{proof}

Now we can proof Theorem \ref{satzBanach}:

\begin{proof}
Again we define $D:=\bigcup_{g\in G} T\lambda_{g}(\mathfrak{h})$. The vector field with parameters $X \colon G \times \frak{h}\rightarrow TG,~ (g,v) \mapsto T\lambda_g(v)$ obviously satisfies the conditions (\ref{BanachBanach1})--(\ref{BanachBanach3}) of Theorem \ref{BanachBanach}. Also condition (\ref{BanachBanach4}) is satisfied, because $\Phi \colon \mathbb{R}\times G \times \frak{h} \rightarrow G,~ (t,g,v) \mapsto \lambda_g(\exp_G(tv))$ is a local flow with parameters of $X$ which follows from
\begin{align*}
\frac{d}{dt}\Big|_{t=0} \Phi(t,g,v)= T\lambda_g(v)= X(g,v).
\end{align*}
\end{proof}

\appendix
\section{Flows of  vector fields on infinite-dimensional manifolds}
Except of the Lemmas \ref{W} and \ref{EinD}, that come from \cite{Gloeckner u. Neeb}, the definitions, theorems and proofs are straightforward generalisations of the well-known case of flows of Banach manifolds as it can be found in \cite[Chapter 1]{Lang} as all ready mentioned in the introduction. We recall this constructions for the convenience of the reader.

\begin{convention}
Throughout this section $E$ will be a locally convex space, $r \in \mathbb{N}$ and $M$ a $C^r$-manifold modelled over $E$.
\end{convention}

\begin{definition}
A $C^1$-curve $\gamma \colon I \rightarrow M$ is called \emph{integral-curve} of a vector field $X \colon M \rightarrow TM$, if $\gamma'(t) = X (\gamma(t))$ for $t\in I$.
\end{definition}

\begin{definition} \label{FL8}
If $\Omega \subseteq \mathbb{R} \times M$ is open and $\Phi\colon\Omega \rightarrow M$, $(t,x) \mapsto \Phi(t,x)=:\Phi_{t}(x)$ is a $C^{r}$-map such that 
\begin{compactenum}[(a)]
\item $\{0\} \times M \subseteq \Omega$,
\item $\Phi_{0}(x)=x$  for all  $x\in M$ and
\item given $x\in M$ we can find an $x$-neighbourhood $U\subseteq M$ and an open symmetric Interval $I$ such that $I \times U \subseteq \Omega$, $I \times \Phi( I\times U) \subseteq \Omega$ and $\Phi_{s}(\Phi_{t}(y))=\Phi_{s+t}(y)$  for $y \in U$ and $s,t \in I$ with $s+t \in I$,\label{3}
\end{compactenum}
then we call $\Phi$ a \emph{local $C^r$-$\mathbb{R}$-action} on $M$. 

\end{definition}

\begin{lemma}\label{UundI}
Given a local $C^r$-$\mathbb{R}$-action $\Phi \colon \Omega \rightarrow M$ and $x\in M$, we can find a pair $(I,U)$ of an open symmetric interval $I\subseteq \mathbb{R}$ and an open $x$-neighbourhood $U \subseteq M$ like in \ref{FL8}.\ref{3} such that $\Phi_t|_U \colon U \rightarrow \Phi_t(U)$ is a diffeomorphism between open subsets of $M$ with the inverse function $\Phi_{-t}|_{\Phi_t(U)} \colon  \Phi_t(U) \rightarrow U$ for all $t\in I$.
\end{lemma}
\begin{proof}
Let $I'\subseteq \mathbb{R}$ be an open symmetric interval and $U' \subseteq M$ be an open $x$-neighbourhood such that Definition \ref{FL8}.\ref{3} holds. It is clear that $\Phi_t|_{U'}$ is injective for all $t\in I'$. 

Let $I\subseteq I'$ be an open symmetric interval and $U \subseteq U'$ be an open $x$-neighbourhood such that $\Phi ( I \times U ) \subseteq U'$. We have $y = \Phi_{-t} \circ \Phi_t(y)$ and $\Phi_t(y) \in U'$ for all $y \in U$ and $t \in I$ and get $\Phi_t(y) = (\Phi_{-t}|_{U'})^{-1}(y)$. We see $\Phi_t(U) = (\Phi_t|_{U'})^{-1}(U)$ and conclude that $\Phi_t(U)$ is open in $U'$ respectively $M$.

It is clear that $\Phi_{-t}|_{\Phi_t(U)} \colon \Phi_t(U) \rightarrow U$ is a $C^r$-inverse function to $\Phi_t|_U$.
\end{proof}

\begin{definition}
A local $C^r$-$\mathbb{R}$-action $\Phi \colon \Omega \rightarrow M$ is called \emph{local flow} of a $C^r$- vector field $X \colon M \rightarrow TM$ if $\frac{\partial}{\partial t}\big|_{t=0} \Phi(t,x) = X(x)$ for all $x\in M$. In this situation we also say that the vector field $X$ {\it provides a local flow}.
\end{definition}

\begin{lemma}
For a $C^{r-1}$-vector field $X \colon M \rightarrow TM$ and a local $C^r$-action $\Phi \colon \Omega \rightarrow M$ the following conditions are equivalent:
\begin{compactenum}[(a)]
\item $\Phi \colon \Omega \rightarrow M$ is a local flow for $X$.\label{Eins}
\item Given $x\in M$ we can choose an open symmetric interval $I\subseteq \mathbb{R}$ and an open $x$-neighbourhood $U\subseteq M$ like Definition \ref{FL8}.\ref{3} and get $\frac{\partial}{\partial s}\big|_{s=t} \Phi(s,y) = X (\Phi(t,y))$ for all $t \in I$ and $y \in U$.\label{Zwei}
\end{compactenum}
\end{lemma}
\begin{proof}
(\ref{Eins})$\Rightarrow$(\ref{Zwei}): We calculate
\begin{align*}
\frac{\partial}{\partial s}\Big|_{s=t} \Phi(s,y) = \frac{\partial}{\partial s}\Big|_{s=0} \Phi(s+t , y) = \frac{\partial}{\partial s}\Big|_{s=0} \Phi_s (\Phi_t(y)) = X(\Phi_t(y)).
\end{align*}

(\ref{Zwei})$\Rightarrow$(\ref{Eins}): We have $\frac{\partial}{\partial t}\big|_{t=0} \Phi(t,x) = X(\Phi_0(x)) = X(x)$.
\end{proof}

The above lemma leads directly to the following theorem.

\begin{theorem}
If $X \colon M \rightarrow TM$ is a $C^{r-1}$-vector field that provides a local flow and $x \in M$, then we can find an integral-curve $\gamma\colon I \rightarrow M$ of $X$ on an open symmetric interval $I$ with $\gamma(0) = x$. 
\end{theorem}

The following Lemmas \ref{W} and \ref{EinD} as well as their proofs come from \cite{Gloeckner u. Neeb}. We only recall them for the convenience of the reader. They wear also cited in \cite[Chapter 1]{Eyni}.
\begin{lemma}\label{W}
Let $\Phi\colon\Omega\rightarrow M$ be a local $C^r$-$\mathbb{R}$-action, $x\in M$, $I$ be an open symmetric interval and  $U \subseteq M$ be an $x$-neighbourhood like in Definition \ref{FL8}.\ref{3}. Because $\Phi$ is continuous we find an open symmetric interval $I'\subseteq I$ and an open $x$-neighbourhood $W\subseteq U$ such that $\Phi(I' \times  W) \subseteq U$. In this situation we have $W\subseteq \Phi_{t}(U)$ for $t\in I'$.
\end{lemma}
\begin{proof}
Given $t\in I'$ and $z\in W$ we have $z\in U$ and so $\Phi_{t}(\Phi_{-t}(z))=\Phi_{0}(z)=z$. With $\Phi_{-t}(z)\in U$ we get $z\in\Phi_{t}(U)$.
\end{proof}

\begin{lemma}\label{EinD}
If $X\colon M\rightarrow E$ is a $C^{r-1}$-vector field that provides a local flow, $L\subseteq \mathbb{R}$ is an interval, $\varphi, \psi \colon L\rightarrow M$ are integral-curves for $X$ and $t_{0}\in L$ with $\varphi(t_{0})=\psi(t_{0})$, then $\varphi=\psi$. 
\end{lemma}
\begin{proof}
Let $\Phi \colon \Omega \rightarrow M$ be a local flow for $X$. Obviously the curves $\varphi$ and $\psi$ are automatically $C^{r}$-curves. We define $T:=\{t\in L: \varphi(t)=\psi(t)\}\subseteq L$. The set $T$ is closed in $L$, because $\varphi$ and $\psi$ are continuous and the diagonal in $M\times M$ is closed. Moreover $T\neq \emptyset$, because $t_{0}\in T$. We get $T=J$, if we can show that $T$ is open, because $J$ is connected. Let $t_{1}\in T$. We have to construct a $\delta>0$ such that $]t_{1}-\delta,t_{1}+\delta[\cap J\subseteq T$.

Let $x_{1}:=\varphi(t_{1})=\psi(t_{1})$. We find an open symmetric interval $I \subseteq \mathbb{R}$ and an open $x_{1}$-neighbourhood $U\subseteq M$ like in Definition \ref{FL8}.\ref{3}. Next we choose  an open symmetric interval $J\subseteq I$, an open $x_1$-neighbourhood $P \subseteq U$ with $\Phi (J \times P ) \subseteq U$, an open symmetric interval $J' \subseteq J$ and an open symmetric $x_1$-neighbourhood $W \subseteq P$ such that $\Phi(J' \times W) \subseteq P$. Now we get $W\subseteq \Phi_t(P)$ for $t \in J'$ with Lemma \ref{W}.

Given $x\in P$ we define $\alpha_x \colon J \rightarrow M$, $\alpha_x(t) = \Phi_{t}(x)$.
For $t\in J$ we get
\begin{align*}
x=\Phi_{-t}(\Phi_{t}(x))=\Phi(-t,\Phi(t,x)).
\end{align*}
Derivation in $t$ on both sides leads to
\begin{align}
&0=\frac{d}{dt}\Phi(-t,\Phi(t,x)) = T\Phi(-\sbull,\Phi(\sbull,x))(t,1)\nonumber\\
=&T\Phi\circ(T(-{\id}_{\mathbb{R}}),T\Phi(\sbull,x))(t,1) =T\Phi((-t,1),\alpha_x'(t))\nonumber\\
=&T\Phi(\sbull,\alpha_x(t))(-t,-1)+T\Phi(-t,\sbull)(\alpha_x'(t))\nonumber\\
=&(-1) \cdot \frac{\partial}{\partial s}\Big|_{s=-t} \Phi_s(\alpha_x(t)) + T(\Phi_{-t})(\alpha_x'(t))\nonumber\\
=&-X(\Phi_{-t}(\alpha_x(t))) + T(\Phi_{-t})(\alpha_x'(t))\label{FL5}
\end{align}
where we deduced $\alpha_x(t)\in V$ from $x\in P$ and $t\in J$ in order to show (\ref{FL5}). Altogether we have shown equation (\ref{FL5}) for $x\in P$ and $t\in J$.

Now let $w\in W$ and $t \in J'$. We find $x\in P$ with $\Phi_{t}(x)=w$.  Like above we define $\alpha_x(s):=\Phi_{s}(x)$ for $s \in J$ and get $0 = -X( \Phi_{-s} (\alpha_x(s))) + T(\Phi_{-s}) ( X( \alpha_x(s) ) )$ for $s \in J$. Hence 
\begin{align}\label{FL7}
0=-X(\Phi_{-t}(w)) + T(\Phi_{-t}) (X(w)),
\end{align}
because $t \in J' \subseteq J$ and $\alpha_x(t)=w$. (\ref{FL7}) holds for $w\in W$ and $t \in J'$.

Because $\varphi$ is continuous, we find $\delta \in J' \cap ]0,\infty[$ such that $\varphi\left(]t_{1}-\delta,t_{1}+\delta[\cap L\right)\subseteq W$. We define the $C^{1}$-map $h\colon]-\delta,\delta[ \rightarrow M$, $h(t) = \Phi_{-t}(\varphi(t+t_{1}))$. Given $t\in]-\delta,\delta[$ we conclude
\begin{align*}
&h'(t)=T\Phi(-\sbull,\varphi(\sbull+t_{1}))(t,1) 
=T\Phi\circ(T(-{\id}_{\mathbb{R}}),T\varphi(\sbull+t_{1}))(t,1)\\
=&T\Phi((-t,-1),\varphi'(t+t_{1}))\\
=&T\Phi(\sbull,\varphi(t+t_{1}))(-t,-1)+T\Phi(-t,\sbull)(\varphi'(t+t_{1}))\\
=&-X(\Phi_{-t}(\varphi(t+t_{1})))+T\Phi(-t,\sbull)(X(\varphi(t+t_{1}))) =0
\end{align*}
where we used $\varphi(t+t_{1})\in W$ and (\ref{FL7}). Thus $h$ is constant, because intervals are connected. We have $h(0)=\varphi(t_{1})$ and so $\Phi_{-t}(\varphi(t+t_{1}))=\varphi(t_{1})$ for all $t\in ]-\delta,\delta[$. After shrinking $\delta$ we assume $\psi(]t_{1}-\delta,t_{1}+\delta[\cap L)\subseteq W$. A similar calculation as above shows
\begin{align*}
\Phi_{-t}(\varphi(t+t_{1}))=\varphi(t_{1})=\psi(t_{1})=\Phi_{-t}(\psi(t+t_{1}))
\end{align*}
for $t\in]-\delta,\delta[$. The map $\Phi_{t}|_{P}$ is injective and $\varphi(t),\psi(t)\in W\subseteq P$. Thus $\varphi|_{]t_{1}-\delta,t_{1}+\delta[}=\psi|_{]t_{1}-\delta,t_{1}+\delta[}$. 
\end{proof}

\begin{definition}
Given a $C^{r-1}$-vector field $X\colon M\rightarrow TM$ that provides a local flow and $x \in M$ we define $I_x := \bigcup_\gamma I_\gamma$, where the union is taken over all integral-curves $\gamma \colon I_\gamma \rightarrow M$ of $X$ with $0\in I_\gamma$ and $\gamma(0) = x$. In this situation
\begin{align*}
\gamma_x \colon I_x \rightarrow M,~ t  \mapsto \gamma(t) \text{ for } t\in I_\gamma
\end{align*}
is a well-defined map and an integral-curve for $X$. We call $\gamma_x$ the \emph{maximal integral-curve} of $X$ that maps $0$ to $x$.
\end{definition}
\begin{proof}
If $t \in I_{\gamma_1} \cap I_{\gamma_2} =: I$, then $\gamma_1(t) = \gamma_2(t)$, because $0,t \in I$ and $\gamma_i|_I \colon I \rightarrow M$ are integral-curves of $X$ with $\gamma_1(0)=\gamma_2(0)$. We get $\gamma_1|_I = \gamma_2|_I$. 
\end{proof}

\begin{lemma}\label{FlussLemma}
Let $X \colon M \rightarrow TM$ be a $C^{r-1}$-vector field, $x \in M$ and $t \in I_x$. We write $y:= \gamma_x(t)$ and get
$I_x-t = I_{y}$ and $\gamma_x(\tau + t) = \gamma_{y}(\tau)$  for $\tau \in I_{y}$.
\end{lemma}
\begin{proof}
Defining $\varphi_1 \colon I_x -t \rightarrow M$, $t \mapsto \gamma_x(\tau + t)$, we have $\varphi_1'(t)= \gamma_x'(\tau + t) = X (\gamma_x(\tau + t)) = X(\varphi_1(\tau))$. From $t \in I_x$ we deduce $0 \in I_x - t$ and $\varphi_1(0) = y$. Hence
\begin{align}
&I_x-t \subseteq I_{y} \text{ and} \label{ofFl1}\\
&\gamma_{y}(\tau) = \varphi_1(\tau) \text{ for } \tau \in I_x-t \label{ofFl2}.
\end{align}
We define $\varphi_2 \colon I_y +t \rightarrow M$, $\tau \mapsto \gamma_y(\tau-t)$. With (\ref{ofFl1}) and $0 \in I_x$ we get $-t \in I_x - t \subseteq I_y$. Therefore $0 \in I_y +t$. With (\ref{ofFl2}) we calculate $\varphi_2(0)=\gamma_y(-t) = \varphi_1(-t) = \gamma_x(0) =x$. Obviously $\varphi_2'(\tau) = X(\varphi_2(\tau))$ for $\tau \in I_y +t$. Thus $I_y + t \subseteq I_x$ and so $I_x-t = I_y$, because of (\ref{ofFl1}). Using (\ref{ofFl1}) again we get $\gamma_y(\tau) = \gamma_x(\tau +t)$.
\end{proof}

\begin{definition}\label{GlobFluss}
Let $X\colon M\rightarrow TM$ be a $C^{r-1}$-vector field  that provides a local flow, $x\in M$ and $\gamma_x \colon I_x \rightarrow M$ be the maximal integral-curve of $X$ that maps $0$ to $x$. We define $\Omega^X := \bigcup_{x\in M} I_x \times \{x\} \subseteq \mathbb{R} \times M$ and call the map $\Phi^X \colon \Omega^X \rightarrow M$, $(x,t) \mapsto \gamma_x(t)$ the \emph{global flow} of the vector field $X$. 
Given $t \in \mathbb{R}$ we also define the set $\Omega^X_t := \{x \in M: (t,x) \in \Omega\} \subseteq M$.
\end{definition}

In the following proof of Theorem \ref{FlussCr} we follow the ideas of the proof of \cite[Theorem 1.16]{Lang} that is formulated for Banach manifolds but works also for manifolds that are modelled over locally convex spaces.
\begin{theorem}\label{FlussCr}
In the situation of Definition \ref{GlobFluss} $\Omega^X$ is open in $\mathbb{R} \times M$ and $\Phi^X$ is a $C^r$-map. Moreover
\begin{align}\label{FlussAequi}
(s, \Phi^X(t,x)) \in \Omega^X \Leftrightarrow (t+s,x) \in \Omega^X,
\end{align}
for $(t,x) \in \Omega^X$ and $s\in \mathbb{R}$ and in this situation
\begin{align}
\Phi^X (s , \Phi^X(t,x)) = \Phi^X (t+s , x).\label{FlussGleich}
\end{align}
\end{theorem}
\begin{proof}
The relation (\ref{FlussAequi}) and the equation (\ref{FlussGleich}) follow directly from Lemma \ref{FlussLemma}. It remains to show the openness of $\Omega$ and the differentiability-assertion of $\Phi$. 

First, we mention a trivial fact. For every $x \in M$ we find a symmetric interval $I$ and an $x$-neighbourhood $V \subseteq M$ such that $I \times V \subseteq \Omega$ and $\Phi$ is of class $C^r$ on $I \times V$. 

Let $x_0 \in M$ and 
\begin{align*} 
Q:= \left\{b>0: \left(\forall ~ 0 \leq t < b\right) ~ \left(\exists ~ (t,x_0)-\text{neighbourhood } W \subseteq \Omega\right) ~ \Phi|_W \text{ is } C^{r} \right\}.
\end{align*}
Obviously $Q \neq \emptyset$. We show $I_{x_0} \cap [0, \infty[ \subseteq Q$. 

First case: $Q \subseteq \mathbb{R}$ has no upper bound. Then obviously the assertion holds.

Second case: $Q$ has an upper bound. Defining $b:= \sup Q$, we show $b >I_{x_0}$. Assume $b \leq I_{x_0}$. Then $b \in I_{x_0}$. Now let $J \subseteq \mathbb{R}$ be an open symmetric interval and $V \subseteq M$ be a $\Phi(b,x_0)$-neighbourhood such that $J \times V \subseteq \Omega$ and $\Phi|_{J \times V}$ is of class $C^{r}$. Because $I_{x_0} \rightarrow M$, $\tau \mapsto \Phi(\tau, x_0)$ is continuous, we find $t_1 \in ]0,b[$ with $\Phi(t_1,x_0) \in V$ and $b-t_1 \in J$. Moreover we find an interval $I$ and a $x_0$-neighbourhood $U \subseteq M$ such that $I \times U \subseteq \Omega$ and $\Phi|_{I \times U}$ is of class $C^{r}$, because $t_1 \in Q$. After shrinking $I$ and $U$, we assume $\Phi(I\times U) \subseteq V$, because $\Phi(t_1,x_0) \in V$.

We define $J':= J+t_1$. For $t \in J'$ and $x \in U$, we have $(t-t_1 , \Phi(t_1,x)) \in J \times V \subseteq \Omega$ and so we get $(t,x)\in \Omega$ with (\ref{FlussAequi}). Hence $J'\times U \subseteq \Omega$. Moreover
\begin{align*}
\Phi(t,x) = \Phi(\underbracket{t-t_1}_{\in J} , \underbracket{\Phi(t_1,x)}_{\in V})
\end{align*}
for $t \in J'$ and $x \in U$ and $\Phi|_{J \times V}$ is a $C^{r}$-map. Thus $J' \times U \subseteq \Omega$ and $\Phi|_{J'\times U}$ is of class $C^{r}$. But we also have $b \in J'$ and $J'$ is open. This is a contradiction to $b = \sup Q$.\\
%
In the analogous way we see $I_{x_0} \cap ]-\infty,0] \subseteq Q$ and conclude $I_{x_0} \subseteq Q$.
\end{proof}


\begin{lemma}
If $X\colon M\rightarrow TM$ is a $C^{r-1}$-vector field that provides a local flow and $t\in \mathbb{R}$, then $\Phi_t^X(\Omega^X_t) = \Omega^X_{-t}$ and $\Phi_t^X \colon \Omega_t^X \rightarrow \Omega_{-t}^X$ is a diffeomorphism between open sets of $M$ with inverse function $\Phi_{-t}^X$.
\end{lemma}
\begin{proof}
It is sufficient to show $\Phi^X_t(\Omega_t^X) = \Omega_{-t}^X$. To this end let $x \in \Omega_t$. From (\ref{FlussAequi}) we know $(-t,\Phi_t(x)) \in \Omega^X \Leftrightarrow (0,x)\in \Omega^X$. Because the latter is true we get the Inclusion "$\subseteq$". Given $(-t, x)\in \Omega^X$ we get $(t, \Phi_{-t}(x))\in \Omega^X$ with (\ref{FlussAequi}). Moreover we have $x = \Phi_t^X(\Phi_{-t}(x))$.  
\end{proof}

\begin{definition}\label{Param}
Let $P$ be a locally convex space.
\begin{compactenum}[(a)]
\item If $\Omega \subseteq \mathbb{R}\times M \times P$ is open and $\Phi \colon \Omega \rightarrow M$ is a $C^r$-map such that 
\begin{compactenum}[(i)]
\item $\{0\} \times M \times P \subseteq \Omega$
\item $\Phi(0,x,p) = x$ for all $x \in M$ and $p \in P$.
\item For all $x \in M$, $p \in P$ we find a symmetric interval $I$, a $x$-neighbourhood $U$ and a $p$-neighbourhood $V$ such that $I \times U \times V \subseteq \Omega$, $I \times \bigcup_{p \in P} \Phi(I \times U \times \{p\}) \times \{p\} \subseteq \Omega$ and $\Phi(t,\Phi(s,x,p),p) = \Phi(t+s,x,p)$,
\end{compactenum}
then we call $\Phi$ a {\it $C^r$-$\mathbb{R}$-action on $M$ with parameters}. Obviously $\Phi \colon \Omega \rightarrow M$ is a  $C^r$-$\mathbb{R}$-action on $M$ with parameters if and only if, $\tilde{\Phi} \colon \Omega \rightarrow M \times P$, $(t,x,p) \mapsto (\Phi(t,x,p),p)$ is a local $C^r$-action on $M \times P$.
\item A $C^{r-1}$-map $X \colon M \times P \rightarrow TM$ is called {\it vector field with parameters}, if $X(\sbull,p) \in \mathcal{V}(M)$ for all $p \in P$.
We say that {\it $X$ provides a local flow with parameters}, if we find a local $C^r$-$\mathbb{R}$-action on $M$ with parameters such that $\frac{\partial}{\partial t}\big|_{t=0} \Phi(t,x,p)=X(x,p)$. Obviously the vector field $X(\sbull,p)$ provides a local flow for all $p \in P$. Moreover it is clear that $X$ provides a local flow with parameters if and only if $\tilde{X} \colon M \times P \rightarrow T(M \times P) = TM \times P \times P$, $(x,p) \mapsto (X(x,p),p,0)$ provides a local flow. 
\item Given is a vector field with parameters $X \colon M \times P \rightarrow M$ that provides a local flow with parameters, we define $I_{x,p} := \bigcup_\gamma I_\gamma$, where the union is taken over all integral-curves $\gamma \colon I_\gamma \rightarrow M$ of $X(\sbull,p)$ with $0\in I_\gamma$ and $\gamma(0) = x$. Then 
$\gamma_{x,p} \colon I_{x,p} \rightarrow M,~ t  \mapsto \gamma(t)$ for  $t\in I_\gamma$ 
is a well-defined map and an integral-curve for $X(\sbull,p)$. We call $\gamma_{x,p}$ the \emph{maximal integral-curve} of $X$ with parameter $p$ that maps $0$ to $x$. Obviously $\gamma\colon I\rightarrow M$ is the maximal integral curve for $X$ with parameter $p$ that maps $0$ to $x$ if and only if $\tilde{\gamma} \colon I \rightarrow M \times P$, $t \mapsto (\gamma(t),p)$ is the maximal integral curve for $\tilde{X}$ that maps $0$ to $(x,p)$.
\item We set $\Omega^X := \bigcup_{x\in M, p \in P} I_{x,p} \times \{x\} \times \{p\} \subseteq \mathbb{R} \times M \times P$ and define the map $\Phi^X \colon \Omega^X \rightarrow M$, $(t,x,p) \mapsto \gamma_{x,p}(t)$ that we call the \emph{global flow with parameters} for the vector field $X$. 

With the Definitions \ref{GlobFluss} and \ref{Param} we conclude that $\Omega^X$ is open in $\mathbb{R} \times M \times P$ and that $\Phi^X$ is a $C^r$-map. Moreover given $(t,x,p) \in \Omega^X$ and $s\in \mathbb{R}$
\begin{align}\label{FlussAequi}
(s, \Phi^X(t,x,p),p) \in \Omega^X \Leftrightarrow (t+s,x,p) \in \Omega^X,
\end{align}
and in this situation
\begin{align}
\Phi^X (s , \Phi^X(t,x,p),p) = \Phi^X (t+s,x,p).\label{FlussGleich}
\end{align}
\end{compactenum}
\end{definition}

%
%

\end{document}